\documentclass[12pt]{amsart}
\usepackage{amsmath}
\usepackage{amsthm}
\usepackage{amssymb}
\usepackage{amscd}
\usepackage{amsfonts}
\usepackage{amsbsy}
\usepackage{epsfig}
\usepackage{graphicx}
\usepackage{tikz}
\usepackage{tikz-cd}
\usetikzlibrary{arrows}
\usepackage{appendix}
\usepackage{pgfplots}
\pgfplotsset{compat=1.18}

\usepackage{pgf}

\usepackage{xcolor}

\newtheorem{theorem}{Theorem}[section]
\newtheorem{corollary}[theorem]{Corollary}
\newtheorem{proposition}[theorem]{Proposition}
\newtheorem{lemma}[theorem]{Lemma}

\newtheorem{claim}{Claim}[theorem]

\theoremstyle{definition}
\newtheorem{remark}[theorem]{Remark}
\newtheorem{question}[theorem]{Question}
\newtheorem{definition}[theorem]{Definition}
\newtheorem{example}[theorem]{Example}
\newtheorem{problem}{Problem}

\def\ie{{\em i.e.,} }
\def\eg{{\em e.g.} }

\newfont\bbf{msbm10 at 12pt}
\def\eps{\varepsilon}
\def\phi{\varphi}
\def\R{{\mathbb R}}

\def\N{{\mathbb N}}
\def\Z{{\mathbb Z}}

\def\red{\textcolor{red}}
\def\blue{\textcolor{blue}}
\def\black{\textcolor{black}}

\def\theta{\vartheta}

\def\Ent{\mbox{\rm Ent}}

\pgfmathdeclarefunction{fu}{1}{
	\pgfmathparse{(#1 < 0.5) * (2 * #1) + (#1 >= 0.5) * (-0.5 * #1 + 1.75)}
}

\begin{document}
	
	\title{Euclidean-type algorithm for functions with same or similar origin}
	
	\author{Ana\ Anu\v si\' c and Christopher Mouron}
	\address[A.\ Anu\v{s}i\'c]{University of Zagreb Faculty of Electrical Engineering and Computing, Department of Applied Mathematics, Unska 3, 10000 Zagreb, Croatia}
	\email{aanusic@fer.unizg.hr}
	\address[C.\ Mouron]{Rhodes College, 2000 North Parkway, Memphis, TN 38112, USA}
	\email{mouronc@rhodes.edu}
	\thanks{AA was supported by the grant HRZZ-IP-2022-10-9820.}
	\date{\today}
	
	
	\begin{abstract}
		Maps $f,g\colon X\to X$ are called kin if they are forward iterates of the same map $\varphi\colon X\to X$, up to a composition with a commuting homeomorphism. Kin form an important class of commuting maps on $X$. In this paper, we characterize kin, and give an Euclidean-type algorithm which tests when two maps $f,g\colon X\to X$ are kin. Furthermore, we compute the topological entropy of diagonal maps induced by commuting diagonal kin diagrams.  
	\end{abstract}

	\maketitle
	
	\section{Introduction}
	
	Functions $f,g \colon X \to X$ are said to \emph{commute} if $f \circ g(x) = g \circ f(x)$ for all $x \in X$. In this paper we introduce and study a special class of commuting functions, which we call \emph{kin}. We work primarily with \emph{continua}--that is, compact, connected metric spaces $X$, and with continuous onto functions on $X$, which we refer to simply as \emph{maps}. This setting is chosen mainly for convenience, as most of our results extend to considerably more general spaces and classes of functions.
	
	The structure of commuting maps remains poorly understood, even on simple spaces such as the unit interval $I=[0,1]$. Pairs of maps rarely commute, and explicit constructions are either trivial, or notoriously difficult. Recently, however, there has been growing interest in developing methods for constructing commuting pairs (in particular, on simple spaces such as trees), motivated in part by the need for examples of dynamical systems on topological spaces that exhibit specific desirable properties. For example, applications of commuting maps on trees appear (see \cite{Bellamy80, OvRog, OvRog2, HoHerGut}) in questions related to the {\em planar fixed point property} problem, which is considered to be one of the most important open problems in planar topology \cite{Bing69, Hagopian07}.
	
	In this paper we continue to pursue a systematic study of commuting maps of the unit interval $I = [0,1]$ that are additionally \emph{piecewise monotone}, initiated in \cite{AM-commuting}. There, we classified all pairs of commuting piecewise monotone interval maps under an additional technical condition, which we called \emph{strong commutativity}. Piecewise monotone, commuting interval maps are arguably the simplest non-trivial maps on continua. Our hope is that classifying such pairs will provide further insight and offer a potential approach toward understanding commutativity of maps on more general spaces (such as trees). 
	
	\begin{problem}
		Classify all piecewise monotone commuting maps $f,g \colon I \to I$. 
		What if we additionally assume that $f$ and $g$ are piecewise linear?
	\end{problem}
	
	The only currently known pairs of commuting piecewise monotone maps on the interval are:
	\begin{enumerate}
		\item pairs of symmetric tent maps $(T_n, T_m)$ (see Figure~\ref{fig:tents}),
		\item pairs $(\varphi,h)$ consisting of a piecewise monotone interval map $\varphi$ and a homeomorphism $h$ that commutes with $\varphi$ (see Figure~\ref{fig:Anna}),
		\item pairs $(f,g)$ that are kin (see the definition below),
	\end{enumerate}
	as well as certain pairs $(f,g)$ that are semi-conjugate to the three classes above (under the same semi-conjugacy), as described in \cite{Folkman}, or \cite[Figure~15]{AM-commuting}, or certain ``diagonal'' or ``anti-diagonal'' combinations of maps from the classes above, as illustrated in \cite[Figures~12--14]{AM-commuting}. More about previous and current research on commuting interval maps can be found in \cite{Ritt, Baxter, Isbell, Boyce, Huneke, Cohen, Joichi, Folkman, McDowell, AliKoo, CanovasLinero, AM-commuting, RFBrown}. Also, other recent research on commuting maps asks about the existence of other common periodic points of commuting interval maps (see \eg \cite{AliKoo}, \cite{CanovasLinero}), or assumes $X$ is a more complicated space, \eg a square $I^2$ (see \cite{GrincSnoha},\cite{Linero}), or  an arc-like continuum (see \cite{Boronski}, \cite{Mou}). Overview of other contemporary research in the area can be found in \cite{RFBrown}.
	
	The first focus of the paper is to understand the properties of commuting maps that we call \emph{kin}. Kin slightly generalize the (obviously) commuting pairs of the form $(\varphi^n,\varphi^m)$, for some some map $\varphi\colon X\to X$ (these pairs are called {\em cousins}). To be more precise, maps $f,g\colon X\to X$ are called {\em kin} if there is a map $\varphi\colon X\to X$, a homeomorphism $h\colon X\to X$ that commutes with $\varphi$, and $a,b\in\N\cup\{0\}$, $n,m\in\N$ such that $f=h^a\circ \varphi^n$, $g=h^b\circ \varphi^m$. Cousins can be interpreted as different states of some dynamical system $(X,\varphi)$, so that cousins $f$ and $g$ can be viewed as maps sharing the same ``origin". Kin do not share the same ``origin" (unless they are cousins), but they possess a similar origin in the sense that the iterates of $\varphi$ are only mildly perturbed by a homeomorphism.
	
	Maps that are kin obviously commute. However, given commuting maps $f,g$, there does not appear to be an obvious way to determine whether they are kin (or cousins). Suppose we (or a computer) find a pair of commuting maps on $X$. How can we decide whether these maps are kin (or cousins), \ie whether they fall into a trivial class of commuting maps? In this paper, we propose an Euclidean-type algorithm (see Theorem~\ref{thm:cousins}) that detects when two commuting maps are kin. For piecewise monotone interval maps, this algorithm always terminates in finitely many steps, which makes its implementation practically useful.
	
	The second focus of the paper is on extending the entropy results we obtained in \cite{AM-entropy}. Given a continuum $X$ and a map $f\colon X \to X$, we form a new continuum, called the \emph{inverse limit} of $(X,f)$, defined by
	$$\varprojlim(X,f) := \{(x_0, x_1, x_2, \ldots) \in X^{\mathbb{N}} : f(x_i) = x_{i-1} \text{ for all } i \in \mathbb{N}\},$$
	considered as a subspace of $X^{\mathbb{N}}$ equipped with the product topology.  
	If $g \colon X \to X$ commutes with $f$, then we may define a ``diagonal" map 
	$G \colon \varprojlim(X,f) \to \varprojlim(X,f)$,  
	$$G((x_0,x_1,x_2,\ldots)) := (g(x_1), g(x_2), \ldots).$$
	Thus we obtain a new dynamical system $(\varprojlim(X,f), G)$, where the space $\varprojlim(X,f)$ is typically a much more complicated continuum than $X$, while the map $G$ is described entirely in terms of maps on the (usually simpler) space $X$. In \cite{AM-entropy}, we showed that if $f,g\colon X\to X$ are piecewise monotone, strongly commuting maps, then the topological entropy, \ie the measure of complexity of the dynamical system $(\varprojlim(X,f), G)$, can be completely understood in terms of the maps $f$ and $g$. Precisely, $\Ent(G)=\Ent(g\circ f^{-1})$, and if $X=I$, then $\Ent(G)=\max\{\Ent(f), \Ent(g)\}$. In this paper we extend these results to diagonal maps $G$ induced by commuting maps which are kin. In Theorem~\ref{thm:ent} we show that, given kin $f=h^a\circ\varphi^n$, $g=h^b\circ\varphi^m$, the entropy of the dynamical system $(\varprojlim(X,f), G)$ is equal to $\Ent(G)=\Ent(h^{b-a}\circ\varphi^{m-n})$, and if $X=I$, then $\Ent(G)=|m-n|\cdot\Ent(\varphi)$.
	
	The paper is organized as follows. In Section~\ref{sec:prelim} we give preliminaries on commuting maps, set-valued functions, topological entropy, and interval maps. In Section~\ref{sec:kin} we define kin and cousins, give examples and a classification theorem, Theorem~\ref{thm:kin}. In Section~\ref{sec:alg} we give an algorithm that tests whether two commuting maps are kin (or cousins). The main result of this section is Theorem~\ref{thm:cousins}, which proves that the algorithm decides when two {\em non-homeomorphisms} on a general continuum $X$ are kin. We give examples, including the examples for which the algorithm never terminates (not piecewise monotone). In Subsection~\ref{subsec:homeo} we extend the algorithm to homeomorphisms on the interval and discuss why we had to restrict to the interval in this case. Finally, in Section~\ref{sec:entropy} we extend results on entropy of diagonal maps from \cite{AM-entropy} to diagonal maps induced by maps $f,g$ which are kin. In Theorem~\ref{thm:ent} we show that, given kin $f=h^a\circ\varphi^n$, $g=h^b\circ\varphi^m$, the entropy of the dynamical system $(\varprojlim(X,f), G)$ can be completely understood in terms of maps $\varphi$ and $h$. 		
	
	\section{Preliminaries}\label{sec:prelim}
	The background of this study comes from the study of topological spaces $X$ which are often additionally assumed to be {\em continua}, \ie compact, connected, metric spaces. We do not need such strong assumptions on spaces $X$ in this paper, except in Subsection~\ref{sec:prel_ent} and Section~\ref{sec:entropy}, where we discuss topological entropy of certain maps on continua. The main results in the first half of the paper actually do not require any additional structure on sets $X$ and the functions between them. However, to keep the paper consistent, we will assume that spaces $X$ are continua, and functions $f\colon X\to X$ are {\em maps}, \ie continuous functions.
	
	We say that two maps $f,g\colon X\to X$ {\em commute} if $f\circ g=g\circ f$. Given a map $f\colon X\to X$, and $x\in X$, we denote the preimage of $x$ under $f$ as $f^{-1}(x)=\{y\in X: f(y)=x\}$.
	
	\begin{lemma}\cite[Lemma~2.1]{AM-commuting}\label{lem:comm}
		Maps $f,g\colon X\to X$ commute if and only if $f(g^{-1}(x))\subseteq g^{-1}(f(x))$ for all $x\in X$.
	\end{lemma}
	
	A {\em set-valued function} on $X$ is a function $F\colon X\to 2^X$, where $2^X$ denotes the set of all {\em non-empty} closed subsets of $X$. We say that $F$ is {\em onto} if for every $y\in X$ there is $x\in X$ such that $y\in F(x)$. If $F\colon X\to 2^X$ is such that $F(x)$ is degenerate (\ie a set containing a single point in $X$) for every $x\in X$, then we say that $F$ is {\em single-valued}. Assume that $F$ is single-valued and denote  by $F(x)=\{f(x)\}$ for every $x\in X$. Then $f\colon X\to X$ is a function, and we often identify $F$ with $f$.
	
	We say that a set-valued function $F\colon X\to 2^X$ is {\em upper semi-continuous} if the {\em graph} of $F$, \ie the set $\Gamma(F)=\{(x,y): x\in X, y\in F(x)\}$, is closed in $X\times X$.
	Assume $f,g\colon X\to X$ are maps such that $g$ is onto. If $X$ is a continuum, then $f\circ g^{-1}\colon X\to 2^X$ is an upper semi-continuous set-valued function on $X$. If $f\circ g^{-1}$ is single-valued, then it is a map. The behavior of a set-valued function $f\circ g^{-1}$ is a central topic in this paper. We begin with some preliminary lemmas.  
	
	\begin{lemma}\label{lem:iterate_map}
		Let $F\colon X\to 2^X$ be a set-valued function such that there is $x\in X$ for which $F(x)$ is non-degenerate. Assume that $F$ is onto. Then, for every $n\in\N$ there exists $x\in X$ such that $F^n(x)$ is non-degenerate.
	\end{lemma}
	\begin{proof}
		Inductively, assume the claim is true for $n\in\N$, and note that $F^n$ is onto. Let $y\in X$ be such that $F(y)$ is non-degenerate, and let $x\in X$ be such that $y\in  F^n(x)$. Then $F^{n+1}(x)=F(F^n(x))\supseteq F(y)$, so $F^{n+1}(x)$ is non-degenerate.
	\end{proof}
	
	\begin{lemma}\label{lem:decompose}
		Let $\varphi\colon X\to X$ be an onto map, and let $a,b\in\Z$ be such that $\varphi^{a+b}\colon X\to 2^X$ is single-valued. Then $\varphi^{a+b}(x)\in\varphi^a(\varphi^b(x))$ for all $x\in X$. 
	\end{lemma}
	\begin{proof}
		If $\varphi$ is a homeomorphism, then obviously $\varphi^{a+b}=\varphi^a\circ\varphi^b$. Assume $\varphi$ is not a homeomorphism, so $\varphi^{-1}\colon X\to 2^X$ is set-valued, onto, and $\varphi^{-n}$ is not single-valued for every $n\in\N$ by Lemma~\ref{lem:iterate_map}. We conclude that $a+b\geq 0$. 
		
		If $a,b\geq 0$, then obviously $\varphi^{a+b}=\varphi^a\circ\varphi^b$, so there is nothing to prove. Assume that either $a>0>b$, or $b>0>a$.
		
		Assume first that $a>0>b$. Let $x\in X$, and let $y\in X$ be such that $\varphi^{-b}(y)=x$ (recall that $\varphi$ is onto). Then $\varphi^{a}(y)=\varphi^{a+b}(\varphi^{-b}(y))=\varphi^{a+b}(x)$, so $\varphi^{a+b}(x)\in\varphi^a(\varphi^b(x))$.
		
		Now assume that $a<0<b$. As before, we can conclude that $\varphi^{a+b}(x)\in\varphi^b(\varphi^a(x))$. Since $\varphi^b$ and $\varphi^{-a}$ commute, by Lemma~\ref{lem:comm}, we have $\varphi^b(\varphi^{a}(x))\subseteq\varphi^a(\varphi^b(x))$, which concludes the proof.
	\end{proof}
	
	\begin{lemma}\label{lem:composing}
		Assume that maps $f,g\colon X\to X$ commute, $g$ is onto, and $a_1, a_2, b_1, b_2\in\N_0$ are such that $f^{a_1}\circ g^{-a_2}\circ f^{b_1}\circ g^{-b_2}\colon X\to X$ is single-valued. Then $f^{a_1}\circ g^{-a_2}\circ f^{b_1}\circ g^{-b_2}=f^{a_1+b_1}\circ g^{-a_2-b_2}$.
	\end{lemma}
	\begin{proof}
		Since $f$ and $g$ commute, then so do $f^{b_1}$ and $g^{a_2}$. By Lemma~\ref{lem:comm}, $f^{b_1}\circ g^{-a_2}(x)\subseteq g^{-a_2}\circ f^{b_1}(x)$ for all $x\in X$. Thus, for $x\in X$,
		$$f^{a_1}\circ g^{-a_2}\circ f^{b_1}\circ g^{-b_2}(x)\supseteq f^{a_1+b_1}\circ g^{-a_2-b_2}(x).$$ 
		Since $g^{a_2+b_2}$ is onto, $f^{a_1+b_1}\circ g^{-a_2-b_2}(x)$ is non-empty for every $x\in X$. Moreover, since $f^{a_1}\circ g^{-a_2}\circ f^{b_1}\circ g^{-b_2}$ is single-valued, then for all $x\in X$,
		$$f^{a_1}\circ g^{-a_2}\circ f^{b_1}\circ g^{-b_2}(x)= f^{a_1+b_1}\circ g^{-a_2-b_2}(x).$$ 
	\end{proof}
	
	\begin{corollary}\label{cor:power}
		Assume that $f, g\colon X\to X$ are commuting maps, $g$ is onto, and $n,m\in\N$ are such that $f^n\circ g^{-m}$ is single-valued. Then, for every $N\in\N$ and every $x\in X$,
		$$(f^n\circ g^{-m})^N(x)=f^{nN}\circ g^{-mN}(x).$$ 
	\end{corollary}
	
	\begin{lemma}\label{lem:commfginv}
		Assume that $f,g,h\colon X\to X$ are maps such that $h$ commutes with both $f$ and $g$, and $f\circ g^{-1}$ is single-valued. Then $h$ also commutes with $f\circ g^{-1}$.
	\end{lemma}
	\begin{proof}
		Let $x\in X$. We have
		$$h\circ f\circ g^{-1}(x)=f\circ h\circ g^{-1}(x)\subseteq f\circ g^{-1}\circ h(x),$$
		since $h$ and $f$ commute, $h$ and $g$ commute, and by Lemma~\ref{lem:comm}. Since $f\circ g^{-1}$ is single-valued, then so are $f\circ g^{-1}\circ h$ and $h\circ f\circ g^{-1}$, and thus $f\circ h\circ g^{-1}(x)= f\circ g^{-1}\circ h(x)$.
	\end{proof}
	
	\begin{lemma}\label{lem:commH}
		Assume that $f,g\colon X\to X$ commute, and there is a map $h\colon X\to X$ such that $f=h\circ g$. Then $f\circ h=h\circ f$, and $g\circ h=h\circ g$.
	\end{lemma}
	\begin{proof}
		We have $g\circ h\circ g=g\circ f=f\circ g=h\circ g\circ g$, which implies $g\circ h=h\circ g$. Therefore, $f\circ h=h\circ g\circ h=h\circ h\circ g=h\circ f$.
		
		Alternatively, we can use Lemma~\ref{lem:commfginv} after noting that $f\circ g^{-1}(x)=\{h(x)\}$ for every $x\in X$, so it is single-valued, and using that $f$ and $g$ commute.
	\end{proof}
	
	\subsection{Preliminaries on entropy}\label{sec:prel_ent}
	Assume that $X$ is a {\em continuum}, \ie a compact, connected, metric space. We denote the metric on $X$ by $d_X$. Recall that a continuous function $f\colon X\to X$ is called a {\em map}. 
	
	Given an upper semi-continuous set-valued function $F\colon X\to 2^X$, and $n\in\N$, an {\em $n$-orbit} is every $n$-tuple $(x_1,\ldots,x_n)\in X^n$ such that $x_{i+1}\in F(x_{i})$ for all $1\leq i<n$. The set of all $n$-orbits of $F$ is denoted by $Orb_n(F)$.
	Given $\eps>0$ and $n\in\N$, a set $S\subset Orb_n(F)$ is called $(n,\eps)$-separated if for every $(x_1,\ldots,x_n), (y_1,\ldots,y_n)\in S$ there is $i\in\{1,\ldots,n\}$ such that $d_X(x_i,y_i)\geq\eps$. By $s_{n,\eps}(F)$ we denote the largest cardinality of an $(n,\eps)$-separated set. The {\em entropy} of $F$ is defined as
	$$\Ent(F)=\lim_{\eps\to 0}\limsup_{n\to\infty}\frac 1n\log(s_{n,\eps}(F)).$$
	
	Positive topological entropy represents exponential growth rate of the distinguishable orbits of $F$, and $\Ent(F)\in[0,+\infty]$. An accessible reference for learning about entropy can be found in Walters \cite{Walters}.  Below we list basic properties of entropy that will be needed later.
	
	\begin{lemma}\cite{Adler}\label{lem:conjugate}
		Given a continuum $X$, and maps $f,g\colon X\to X$, we say that $f$ and $g$ are {\em conjugate} if there is a homeomorphism $h\colon X\to X$ such that $f\circ h=h\circ g$. If $f$ and $g$ are conjugate, then $\Ent(f)=\Ent(g)$.
	\end{lemma}
	\begin{lemma}\cite{Adler}\label{lem:power}
		Given a continuum $X$, a map $f\colon X\to X$, and $n\in\N$,
		$$\Ent(f^n)=n\cdot\Ent(f).$$
	\end{lemma}
	\begin{lemma}\cite[Corollary~3.6]{KelTen}\label{lem:inverse}
		Given a continuum $X$ and an upper semi-continuous $F\colon X\to 2^X$, $\Ent(F^{-1})=\Ent(F)$.
	\end{lemma}
	
	Given a continuum $X$ and a map $f\colon X\to X$, we define the {\em inverse limit} $\varprojlim(X,f)$ as
	$$\varprojlim(X,f):=\{(x_0,x_1,x_2,\ldots): x_i\in f^{-1}(x_{i-1}), i\in\N\}\subset \prod_{i\geq 0}X.$$
	The space $\varprojlim(X,f)$, equipped with the product topology, is also a continuum.
	
	\begin{theorem}\cite[Theorem~3.1]{Ye}\label{thm:Ye}
		Let $X$ be a continuum, and let $f,g\colon X\to X$ be commuting maps on $X$. Let $\Psi\colon\varprojlim(X,f)\to\varprojlim(X,f)$ be defined as
		$$\Psi((x_0,x_1,x_2,\ldots)=(g(x_0),g(x_1),g(x_2),\ldots).$$
		Then $\Ent(\Psi)=\Ent(g)$.
	\end{theorem}
	
	In \cite{AM-entropy} we started the study of so-called {\em diagonal} maps on $\varprojlim(X,f)$. Given a continuum $X$, and $f,g\colon X\to X$ which commute, the $g$-diagonal map $\psi\colon\varprojlim(X,f)\to\varprojlim(X,f)$ is defined as
	$$\Psi((x_0,x_1,x_2,\ldots)=(g(x_1),g(x_2),\ldots).$$
	We then defined $f,g\colon X\to X$ to {\em strongly commute} if $g\circ f^{-1}(x)=f^{-1}\circ g(x)$ for all $x\in X$. With this assumption, we proved the following theorem.
	
	\begin{theorem}\cite[Proposition~3.5 and Corollary~5.5]{AM-entropy}\label{thm:str_comm}
		Let $X$ be a continuum, and assume that maps $f,g\colon X\to X$ strongly commute. Then the entropy of $g$-diagonal map $\Psi\colon\varprojlim(X,f)\to\varprojlim(X,f)$ is
		$$\Ent(\Psi)=\Ent(g\circ f^{-1}).$$
		Moreover, if $X=I=[0,1]$ is the unit interval, and $f,g\colon I\to I$ are piecewise monotone (see below) and strongly commuting maps, then 
		$$\Ent(\Psi)=\Ent(g\circ f^{-1})=\max\{\Ent(f), \Ent(g)\}.$$
	\end{theorem}
	
	In Section~\ref{sec:entropy} we extend this result by computing the entropy of $g$-diagonal maps on $\varprojlim(X,f)$, given certain commuting maps $f,g\colon X\to X$ which commute, but do not strongly commute. 
	
	\subsection{Preliminaries on interval maps}
	We denote the unit interval by $I=[0,1]$. A map $f\colon I\to I$ is called {\em piecewise monotone} if there exist points $0=c_0<c_1<\ldots<c_{n-1}<c_{n}=1$ such that $f|_{[c_i,c_{i+1}]}$ is one-to-one for every $i\in\{0,\ldots,n-1\}$. We assume that $\{c_0,\ldots,c_{n}\}$ is chosen such that $f|_{[c_{i-1},c_{i}]}$ and $f|_{[c_{i},c_{i+1}]}$ are alternating being strictly increasing and strictly decreasing for all $i\in\{1,\ldots,n-1\}$. In that case, points $c_0,\ldots,c_n$ are called {\em critical points}, and we let $C(f)=n$.
	
	\begin{theorem}\cite{MisSl}\label{thm:Mis}
		Let $f\colon I\to I$ be piecewise monotone map. Then
		$$\Ent(f)=\lim_{n\to\infty}\frac 1n\log(C(f^n)).$$
	\end{theorem}
	
	\begin{lemma}\label{lem:entropy} If $\phi,h:I\to I$ are commuting maps such that $\phi$ is piecewise monotone and $h$ is a homeomophism, then $\mbox{Ent}(\phi)=\mbox{Ent}(h\circ\phi)$.
	\end{lemma}
	\begin{proof}
		Note that $(h\circ\varphi)^n=h^n\circ\varphi^n$. Since $h$ is a homeomorphism, $C(h^n\circ\varphi^n)=C(\varphi^n)$. Theorem~\ref{thm:Mis} implies that $\Ent(h\circ\varphi)=\Ent(\varphi)$.
	\end{proof}
	
	\section{Kin and cousins}\label{sec:kin}
	
	In this and the following section we let $X$ be a continuum, and functions on $X$ are assumed to be maps, \ie continuous. However, the results hold in much higher generality, \ie we can simply take $X$ to be a set, and functions on $X$ with no additional properties. Homeomorphisms can be replaced by bijections.  
	
	\begin{definition}
		Let $f,g\colon X\to X$ be maps. We say that $f,g\colon X\to X$ are {\em kin} if there exist a map $\varphi\colon X\to X$, a homeomorphism $h\colon X\to X$ which commutes with $\varphi$, and $a,b\geq 0, n,m\in\N$ such that
		$$f=h^a\circ\varphi^n, \textrm{ and} \quad g=h^b\circ \varphi^m.$$
	\end{definition}
	
	\begin{definition}
		Let $f,g\colon X\to X$ be maps. We will say that $f$ and $g$ are \emph{cousins} if there is a map $\varphi\colon X\to X$, and there are $n,m\in\N$ such that
		$$f=\varphi^n, \textrm{ and} \quad g=\varphi^m.$$
		The map $\varphi$ is called an \emph{ancestor} of $f$ and $g$. Note that if $f,g$ are cousins, then they are kin.
	\end{definition}
	
	\begin{remark}\label{rem:homeo}
		Every two commuting homeomorphisms $f,g\colon X\to X$ are kin. We can take $h=f\circ g^{-1}$ and $\phi=g$. Then $h$ and $\phi$ commute, $f=h\circ\phi$, and $g=h^0\circ\phi$. This gives a variety of examples of pairs of kin which are not necessarily cousins. For a less trivial example, see Example~\ref{ex:Anna}.
	\end{remark}
	
	Given cousins $f,g\colon X\to X$, we can define their {\em greatest ancestor} as a map $\hat\varphi\colon X\to X$ that is an ancestor of $f$ and $g$, and such that for every $\psi\colon X\to X$ that is an ancestor of $f$ and $g$ there exists $L\in\N$ such that $\psi^L=\hat\varphi$. It is not, a priori, clear whether two cousins have a greatest ancestor, since there can be many different maps $\varphi\colon X\to X$ for which there is $n\in\N$ such that $\varphi^n=f$. However, a greatest ancestor always exists as we show in the following proposition.
	
	\begin{proposition}
		Let $f,g\colon X\to X$ be cousins. Then there exists 
		a greatest ancestor of $f$ and $g$.
	\end{proposition}
	\begin{proof}
		Assume $\varphi\colon X\to X$ is a map, and $n,m\in\N$ are such that $f=\varphi^n, g=\varphi^m$. Let $k=gcd(n,m)$. Then there are relatively prime $n',m'\in\N$ such that $n=n'k, m=m'k$. We will show that $\varphi^k$ is a greatest ancestor of $f$ and $g$.
		
		Assume that $\psi\colon X\to X$ is a map, and $i,j\in\N$ are such that $f=\psi^i$, $g=\psi^j$. 
		
		Since $n',m'$ are relatively prime, there exist $\alpha,\beta\in\Z$ such that $\alpha n'+\beta m'=1$. Actually, for every $l\in\Z$, $(\alpha+lm')n'+(\beta-ln')m'=\alpha n'+\beta m'=1$. Choose $l\in\Z$ such that
		$$L:=\alpha i+\beta j + l(m'i-n'j)>0, \textrm{\ie } L\in\N.$$
		
		For all $x\in X$, since we do not know whether $\alpha+lm'>0$ or $\beta-ln'>0$, we use Lemma~\ref{lem:decompose} to conclude:
		\begin{equation*}
			\begin{split}
				\varphi^k(x) & = \varphi^{k((\alpha+lm')n'+(\beta-ln')m')}(x)\in \varphi^{kn'(\alpha + lm')}(\varphi^{km'(\beta-ln')}(x)) \\ & = f^{\alpha+lm'}(g^{\beta-ln'}(x))=\psi^{i(\alpha+lm')+j(\beta-ln')}(x)=\psi^{\alpha i + \beta j + l(m'i-n'j)}(x)=\psi^L(x).
			\end{split}
		\end{equation*}  
		Since $L\in\N$, $\psi^L(x)$ is degenerate, and we conclude $\varphi^k(x)=\psi^{L}(x)$ for all $x\in X$, \ie $\varphi^k=\psi^{L}$. Thus, $\phi^k$ is a greatest common ancestor of $f$ and $g$.
	\end{proof}
	
	We proceed with first characterization of maps that are kin (or cousins).
	
	\begin{theorem}\label{thm:kin}
		Onto maps $f,g\colon X\to X$ are kin if and only if:
		\begin{enumerate}
			\item $f$ and $g$ commute, and
			\item there exist relatively prime $\alpha,\beta\in\N$ and a homeomorphism $h'\colon X\to X$ which commutes with $f$ and $g$, and for which $f^{\alpha}=h'\circ g^{\beta}$.
		\end{enumerate}
		Onto maps $f,g\colon X\to X$ are cousins if and only if:
		\begin{enumerate}
			\item $f$ and $g$ commute, and
			\item there are relatively prime $\alpha, \beta\in\N$ such that $f^{\alpha}=g^{\beta}$.
		\end{enumerate}
	\end{theorem}
	\begin{proof}
		Assume that $f$ and $g$ are kin (respectively, assume $f$ and $g$ are cousins). Let $\varphi\colon X\to X$ be a map, $h\colon X\to X$ be a homeomorphism which commutes with $\varphi$ (respectively, we take $h=id$ if we $f,g$ are cousins), and $a,b\geq 0$, $n, m\in\N$ be such that $f=h^a\circ \varphi^{n}$, and $g=h^b\circ\varphi^m$. Note first that, since $h$ and $\varphi$ commute, then $f$ and $g$ commute, and $h$ commutes with both $f$ and $g$. 
		
		Let $k\in\N$ be the greatest common divisor of $n$ and $m$, so that there are relatively prime $n',m'\in\N$ such that $n=n'k, m=m'k$. Then
		$$f^{m'}=(h^a\circ \varphi^n)^{m'}=h^{am'}\circ \varphi^{nm'}=h^{am'-bn'}\circ h^{bn'}\circ \varphi^{mn'}=h^{am'-bn'}\circ (h^b\circ\varphi^m)^{n'}=h'\circ g^{n'},$$
		where $h'=h^{am'-bn'}$ is a homeomorphism. Now, since both $h$ and $h^{-1}$ commute with $f$ and $g$, so does $h'$. (Respectively, if $f,g$ are cousins, so $h=id$, then $h'=id$, so $f^{\alpha}=g^{\beta}$.)
		
		For the other direction, assume that $f$ and $g$ commute, $\alpha,\beta\in\N$ are relatively prime, and $h'\colon X\to X$ is a homeomorphism which commutes with $f$ and $g$ such that $f^{\alpha}=h'\circ g^{\beta}$ (respectively, we take $h'=id$). Let $\tilde a,\tilde b\in\Z$ be such that
		$$\alpha \tilde a+ \beta \tilde b = 1,$$
		and define $\varphi\colon X\to 2^X$ as:
		$$\varphi=\begin{cases}
			f^{\tilde b}\circ g^{\tilde a}, \quad {\tilde a}<0<{\tilde b},\\
			g^{\tilde a}\circ f^{\tilde b}, \quad {\tilde b}<0<{\tilde a}.
		\end{cases}.$$
		
		Note that $\varphi$ is indeed a set-valued function on $X$ since we assumed that $f$ and $g$ are onto.
		
		Assume that ${\tilde a}<0<{\tilde b}$. Since $f^{\alpha}$ and $g^{\beta}$ commute, and $f^{\alpha}=h'\circ g^{\beta}$, Lemma~\ref{lem:commH} implies that $f^{\alpha}\circ h'=h'\circ f^{\alpha}$, and $g^{\beta}\circ h'=h'\circ g^{\beta}$. Then, Corollary~\ref{cor:power} implies 
		$$(f^{\tilde b}\circ g^{\tilde a})^{\alpha}(x)=f^{{\tilde b}\alpha}\circ g^{{\tilde a}\alpha}(x)=(h')^{{\tilde b}}\circ g^{{\tilde b}\beta+{\tilde a}\alpha}(x)=(h')^{{\tilde b}}\circ g(x),$$
		and, since $g^{\beta}=(h')^{-1}\circ f^{\alpha}$, and thus by Lemma~\ref{lem:commH}, $(h')^{-1}$ commutes with $f^{\alpha}$. Then, by Corollary~\ref{cor:power},
		$$(f^{\tilde b}\circ g^{\tilde a})^{\beta}(x)=f^{{\tilde b}\beta}\circ g^{{\tilde a}\beta}(x)=f^{{\tilde b}\beta}\circ (h')^{-{\tilde a}}\circ f^{{\tilde a}\alpha}(x)=f^{{\tilde a}\alpha+{\tilde b}\beta}\circ (h')^{-{\tilde a}}(x)=f(x)\circ (h')^{-{\tilde a}},$$
		for all $x\in X$. 
		
		Thus, for $\varphi:=f^{\tilde b}\circ g^{\tilde a}$, Lemma~\ref{lem:iterate_map} implies $\varphi$ is a map, and since $h'$ commutes with both $f$ and $g$, by Lemma~\ref{lem:commfginv}, we have 
		$$g=(h')^{-\tilde b}\circ \varphi^{\alpha},$$
		$$f=\varphi^{\beta}\circ (h')^{\tilde a}=(h')^{\tilde a}\circ \varphi^{\beta}.$$
		Letting $h=(h')^{-1}$, and $b=\tilde b$, $a=-\tilde a$, it follows that $f$ and $g$ are kin. (Respectively, if $h'=id$, then $h=id$, so it follows that $f=\varphi^{\alpha}$, $g=\varphi^{\beta}$, \ie $f$ and $g$ are cousins). The proof is similar when ${\tilde b}<0<{\tilde a}$.
	\end{proof}

	\section{Euclidean algorithm for maps}\label{sec:alg}
	We first recall a version of the Euclidean algorithm for determining the greatest common divisor of positive integers $n, m\in\N$, $gcd(n,m)$. Given positive integers $n>m$, we define
	$$n_1=n, m_1=m.$$
	For every $i\in\N$ such that $n_i\neq m_i$:
	\begin{enumerate}
		\item if $n_{i}>m_{i}$, we define
		$$n_{i+1}=n_{i}-m_{i}, \quad m_{i+1}=m_{i},$$ 
		\item if $m_{i}>n_{i}$, we define
		$$n_{i+1}=n_{i}, \quad m_{i+1}=m_{i}-n_{i}.$$ 
	\end{enumerate}
	If $N\in\N$ is such that $n_N=m_N$, then $n_M=m_N=gcd(n,m)$.
	
	\begin{example}
		The Euclidean algorithm for $n=34, m=10$ gives
		$$(34,10)\to (24,10)\to (14,10)\to (4,10)\to (4,6)\to (4,2)\to (2,2),$$
		so $gcd(34,10)=2$.
	\end{example}
	
	\begin{lemma}\label{lem:terminate}
		The Euclidean algorithm for $n>m$ terminates in at most $n$ steps. It will take exactly $n$ steps if and only if $m=1$ or $m=n-1$.
	\end{lemma}
	\begin{proof}
		The proof follows inductively on $n\in\N$. If $n=2$, $m=1$, then $n_2=1$, $m_2=1$, so the algorithm terminates in $n$ steps. Assume that for some $N\in\N$ and all $n\leq N$, the algorithm terminates in at most $n$ steps. Let $n+1\in\N$, and let $m\in\N$, $m<n+1$. Then $n_1=n+1$, $m_1=m$, and $n_2=n+1-m$, $m_2=m$. Since $m\geq 1$, $n_2\leq n$. Since also $m_2=m\leq n$, the induction hypothesis implies that the algorithm will terminate in at most $n+1$ steps.
		
		It is easy to check that the algorithm terminates in exactly $n$ steps if $m=1$ or $m=n-1$. To see that in all other cases the algorithm will terminate in less than $n$ steps, let $n>m$ and $1<m<n-1$. Then $n_1=n$, $m_1=m$, and $n_2=n-m$, $m_2=m$. If $n_2\geq m_2$, the algorithm will terminate in at most $n-m+1<n$ steps, and if $n_2<m_2$, the algorithm will terminate in at most $m+1<n$ steps.
	\end{proof}
	
	The goal of this section is to give a similar algorithm which determines if two onto maps are kin, or cousins, and finds their ancestor.
	To help with this, we first define a relation $\succ$ on the set of onto maps:
	
	\begin{definition}
		Let $f, g\colon X\to X$ be onto maps. We we will write $f\succ g$ if $g^{-1}\circ g(x)\subseteq f^{-1}\circ f(x)$ for all $x\in X$, and there is $x\in X$ such that $g^{-1}\circ g (x)\neq f^{-1}\circ f(x)$. 
	\end{definition}
	Now, we will consider some lemmas.
	
	\begin{lemma}\label{lem:fnfs}
		Let $f,g\colon X\to X$ be onto maps such that $g^{-1}\circ g(x)\subseteq f^{-1}\circ f(x)$ for all $x\in X$. Then there exists a map $\psi\colon X\to X$ such that $f=\psi\circ g$. If $f, g$ additionally commute, then $g$ and $\psi$ also commute. 
	\end{lemma}
	\begin{proof}
		Assume $g^{-1}\circ g(x)\subseteq f^{-1}\circ f(x)$ for all $x\in X$, and let $\psi\colon X\to 2^X$ be given as $\psi(x)=f\circ g^{-1}(x)$. We will show that $\psi$ is a map. Let $x\in X$, and $t_1, t_2\in X$ be such that $g(t_1)=g(t_2)=x$. Then $t_2\in g^{-1}\circ g(t_1)\subseteq f^{-1}\circ f(t_1)$, so $f(t_2)=f(t_1)$. It follows that $f\circ g^{-1}(x)$ is degenerate, so $\psi$ is a map.
		
		Assume that $f$ and $g$ commute. Then $g\circ \psi=g\circ f\circ g^{-1}=f\circ g\circ g^{-1}=f$. Moreover, for every $x\in X$, $\psi\circ g(x)=f\circ g^{-1}\circ g(x)\supseteq f(x)$. Since $\psi\circ g$ is a map, $\psi\circ g(x)=f(x)=g\circ \psi(x)$.
	\end{proof}
	
	\begin{lemma}\label{lem:fnfmap}
		For onto maps $f, g\colon X\to X$, $f\circ g^{-1}$ is a map if and only if $g^{-1}\circ g(x)\subseteq f^{-1}\circ f(x)$ for all $x\in X$.
	\end{lemma}
	\begin{proof}
		The only if direction follows from Lemma~\ref{lem:fnfs}. For the other direction, assume $f\circ g^{-1}$ is a map, and let $x\in X$. Let $y\in g^{-1}\circ g(x)$, \ie $g(y)=g(x)=t$. Since $f\circ g^{-1}(t)\supseteq\{f(x),f(y)\}$, and $f\circ g^{-1}(t)$ is degenerate, it follows that $f(y)=f(x)$, \ie $y\in f^{-1}\circ f(x)$. 
	\end{proof}
	
	\begin{corollary}
		Assume $f,g\colon X\to X$ are onto maps. Then $f^{-1}\circ f(x)=g^{-1}\circ g(x)$ for all $x\in X$ if and only if there exists a homeomorphism $h\colon X\to X$ such that $f=h\circ g$. If $f$ and $g$ additionally commute, then there exists a homeomorphism $h\colon X\to X$ such that $f=h\circ g=g\circ h$.
	\end{corollary}
	\begin{proof}
		Since $g^{-1}\circ g(x)\subseteq f^{-1}\circ f(x)$ for all $x\in X$, there is $h\colon X\to X$ such that $f=h\circ g$. If additionally $f^{-1}\circ f(x)\subseteq g^{-1}\circ g(x)$ for all $x\in X$, then $g=\tilde h\circ f$, where $\tilde h=g\circ f^{-1}$ is a map. Note that $f=h\circ g=h\circ \tilde h\circ f$, so $h\circ\tilde h=id$, which implies that $h$ is a homeomorphism. The additional part follows Lemma~\ref{lem:fnfs}.
	\end{proof}
	
	\begin{lemma}\label{lem:hkmap}
		Let $\psi\colon X\to X$ be an onto map which is not a homeomorphism, and let $k\in\Z$. Then $\psi^{k}$ is a map if and only if $k\geq 0$.
	\end{lemma}
	\begin{proof}
		If $k\geq 0$, then $\psi^k$ is obviously a map. Assume that $k<0$. Since $\psi$ is not a homeomorphism, there are $x,y\in X$, $x\neq y$, such that $\psi(x)=\psi(y)$, and consequently $\psi^{-k}(x)=\psi^{-k}(y)=z$. Thus, $x,y\in \psi^k(z)$, which implies that $\psi^k$ is not a map.
	\end{proof}

	\begin{remark}
		By Lemma~\ref{lem:fnfmap}, $f\succ g$, if and only if $f\circ g^{-1}$ is a map, and $g\circ f^{-1}$ is not a map, if and only if $f\circ g^{-1}$ is a map which is not a homeomorphism. Moreover, $f\succ g$ immediately implies that $f$ is not a homeomorphism. 
	\end{remark}

	\begin{lemma}\label{lem:hnhm2}
		Let $\varphi\colon X\to X$ be an onto map which is not a homeomorphism, let $h\colon X\to X$ be a homeomorphism which commutes with $\varphi$, and let $a,b\in\Z$, $a,b\geq 0$, and $n,m\in\N$. Then $h^a\circ \varphi^n\succ h^b\circ \varphi^m$ if and only if $n>m$.
	\end{lemma}
	\begin{proof}
		Note that $h^a\circ \varphi^n\succ h^b\circ \varphi^m$, if and only if $h^{a-b}\circ\varphi^{n-m}$ is a map, and $h^{b-a}\circ \varphi^{m-n}$ is not a map. By Lemma~\ref{lem:hkmap}, since $h^k$ is a map for every $k\in\Z$, this is true if and only if $n-m\geq 0$, and $m-n<0$. This implies that $n>m$. 
	\end{proof}
	
	\begin{lemma}\label{lem:inlaws}
		Assume that $f,g\colon X\to X$ are commuting onto maps such that $f\neq g$, and at least one of $f$, $g$ is not a homeomorphism. If $f\circ g^{-1}$ and $g\circ f^{-1}$ are both maps (\ie, $f\circ g^{-1}$ is a homeomorphism), then $f$ and $g$ are kin which are not cousins. 
	\end{lemma}
	\begin{proof}
		Assume that $h\colon X\to X$ is a homeomorphism such that $f\circ g^{-1}=h$, \ie $f=h\circ g$. Since $f$ and $g$ commute, Lemma~\ref{lem:commH} implies that $h$ commutes with both $f$ and $g$. Theorem~\ref{thm:kin} implies that $f$ and $g$ are kin.
		
		Assume there is a map $\varphi\colon X\to X$ and there are natural numbers $n>m$ such that $f=\varphi^n$, $g=\varphi^m$. Since $f\circ g^{-1}$ is a map, $f\circ g^{-1}=\varphi^{n-m}$ is a map, and since $g\circ f^{-1}$ is a map , $g\circ f^{-1}=\varphi^{m-n}$ is a map. However, since at least one of $f,g$ is not a homeomorphism, neither is $\varphi$, so $\varphi^{m-n}$ is not a map. That is a contradiction. 
	\end{proof}
	
	\begin{example}
		Let $\varphi,h\colon X\to X$ be onto maps which commute, where $\varphi$ is not a homeomorphism, and $h\neq id$ is a homeomorphism (see, \eg Example~\ref{ex:Anna} for a particular example on $X=[0,1]=I$). Let $f=h^2\circ\varphi$, $g=h\circ\varphi$. Then $f$ and $g$ are commuting onto maps, $f\neq g$ since $h\neq id$, and $f$ and $g$ are not homeomorphisms. Moreover, $f\circ g^{-1}=h$, $g\circ f^{-1}=h^{-1}$ are both homeomorphisms. Lemma~\ref{lem:inlaws} implies that $f$ and $g$ are kin which are not cousins.
	\end{example}
	
	We now give an algorithm, similar to the Euclidean algorithm,  which determines if maps $f,g$ are kin (cousins), and finds maps $h$ and $\varphi$ (or an ancestor if they are cousins). We will assume that $f,g\colon X\to X$ are onto {\bf non-homeomorphisms}. Since every pair of commuting homeomorphisms are kin, this would be redundant in case when $f$ and $g$ are commuting homeomorphisms. Nevertheless, the question when two homeomorphisms are cousins is still interesting. This will be considered later.
	
	Assume that $f,g\colon X\to X$ are onto non-homeomorphisms which commute. We define
	$$\lambda_1=f, \quad \rho_1=g.$$
	For every $i\in\N$ such that $\lambda_i\succ \rho_i$, or $\rho_i\succ \lambda_i$:
	\begin{enumerate}
		\item if $\lambda_i\succ \rho_i$, we define
		$$\lambda_{i+1}=\lambda_{i}\circ\rho^{-1}_{i}, \quad \rho_{i+1}=\rho_{i},$$ 
		\item if $\rho_i\succ\lambda_i$, we define
		$$\lambda_{i+1}=\lambda_{i}, \quad \rho_{i+1}=\rho_{i}\circ\lambda^{-1}_{i}.$$ 
	\end{enumerate}
	If $N\in\N$ is such that $\lambda_N\circ\rho_N^{-1}$ is a homeomorphism $h$, then $f$ and $g$ are kin with $h$ and $\varphi=\lambda_N=h\circ\rho_N$. If $N\in\N$ is such that $\lambda_N=\rho_N$, then $\varphi=\lambda_N=\rho_N$ is an ancestor of $f$ and $g$. We prove and discuss this in the rest of this section.
	
	\begin{remark}
		It can happen that $\lambda_i, \rho_i$ are not defined for some $i\geq 1$, in which case $\lambda_k, \rho_k$ are not defined for all $k\geq i$. This happens if: there is $i>1$ such that $\lambda_j, \rho_j$ are defined for all $1\leq j<i$, and $\lambda_{i-1}\circ \rho_{i-1}^{-1}$ and $\rho_{i-1}\circ \lambda_{i-1}^{-1}$ are either both maps (in fact, homeomorphisms), or neither of them is a map (see example below).
	\end{remark}
	
	\begin{example}\label{ex:tents}
		Let $\lambda_1=f=T_6$ and $\rho_1=g=T_3$, (see Figure~\ref{fig:tents}). 
		\begin{figure}[ht!]
			\centering
			\begin{tikzpicture}[scale=4]
				\draw (0,0)--(0,1)--(1,1)--(1,0)--(0,0);
				\draw (0,0)--(1/2,1)--(1,0);
				\node[above] at (1/2,1) {\small $T_2$};
			\end{tikzpicture}
			\begin{tikzpicture}[scale=4]
				\draw (0,0)--(0,1)--(1,1)--(1,0)--(0,0);
				\draw (0,0)--(1/3,1)--(2/3,0)--(1,1);
				\node[above] at (1/2,1) {\small $T_3$};
			\end{tikzpicture}
			\begin{tikzpicture}[scale=4]
				\draw (0,0)--(0,1)--(1,1)--(1,0)--(0,0);
				\draw (0,0)--(1/6,1)--(2/6,0)--(3/6,1)--(4/6,0)--(5/6,1)--(1,0);
				\node[above] at (1/2,1) {\small $T_6$};
			\end{tikzpicture}
			\caption{Symmetric tent maps $T_2, T_3, T_6$.}
			\label{fig:tents}
		\end{figure}
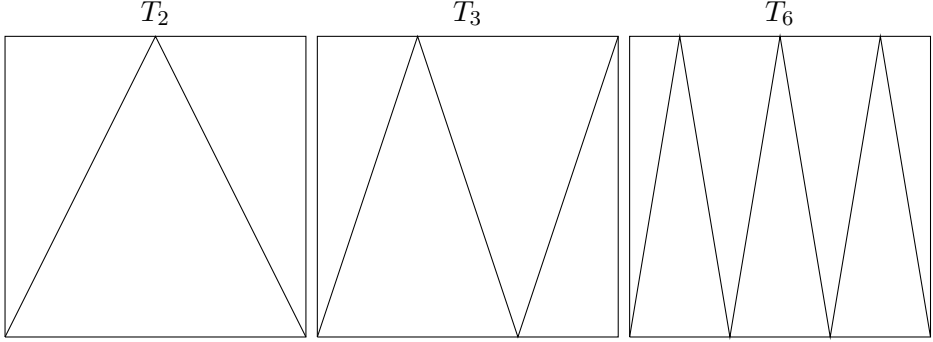
		Then $f\circ g^{-1}=T_2$, and $g\circ f^{-1}=T_2^{-1}$, so $\lambda_1\succ \rho_1$. We can define $\lambda_2=\lambda_1\circ \rho_1^{-1}=T_2$ and $\rho_2=\rho_1=T_3$. However, $\lambda_2\circ \rho_2^{-1}=T_2\circ T_3^{-1}$ and $\rho_2\circ \lambda_2^{-1}=T_3\circ T_2^{-1}$ are not maps, see Figure~\ref{fig:fnfs}. Thus, $\lambda_k, \rho_k$ are not defined $k> 2$. We conclude that $f$ and $g$ are not kin. 
	\end{example}
	
	\begin{figure}
		\begin{tikzpicture}[scale=4]
			\draw (0,0)--(0,1)--(1,1)--(1,0)--(0,0);
			\draw (0,0)--(1,2/3)--(1/2,1)--(0,2/3)--(1,0);
			\node[above] at (1/2,1) {\small $T_2\circ T_3^{-1}$};
		\end{tikzpicture}
		\hspace{5pt}
		\begin{tikzpicture}[scale=4]
			\draw (0,0)--(0,1)--(1,1)--(1,0)--(0,0);
			\draw (0,0)--(2/3,1)--(1,1/2)--(2/3,0)--(0,1);
			\node[above] at (1/2,1) {\small $T_3\circ T_2^{-1}$};
		\end{tikzpicture}
		\caption{Set-valued functions $\lambda_2\circ\rho_2^{-1}$ and $\rho_2\circ\lambda_2^{-1}$ from Example~\ref{ex:tents}. Since none of the set-valued functions is single valued, $\lambda_3$ and $\rho_3$ are not well-defined maps and the algorithm terminates.}
		\label{fig:fnfs}
	\end{figure}
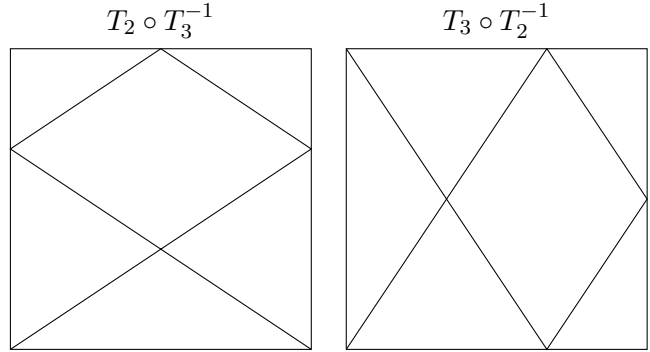

	Maps $\lambda_i, \rho_i$ (when defined), will be of the form $\lambda_i=f^{\ell_f(i)}\circ g^{\ell_g(i)}, \rho_i=g^{r_g(i)}\circ f^{r_f(i)}$, so we can extend the algorithm to include the sequences $(\ell_f(i)), (\ell_g(i)), (r_f(i)), (r_g(i))$ as follows:
	$$\ell_f(1)=1, \quad \ell_g(1)=0; \quad r_f(1)=0, \quad r_g(1)=1.$$
	For every $i\in\N$ such that $\lambda_i\succ \rho_i$, or $\rho_i\succ \lambda_i$:
	\begin{enumerate}
		\item if $\lambda_i\succ \rho_i$, we define
		$$\ell_f(i+1)=\ell_f(i)-r_f(i), \quad \ell_g(i+1)=\ell_g(i)-r_g(i); \quad r_f(i+1)=r_f(i), \quad r_g(i+1)=r_g(i).$$
		\item if $\rho_i\succ\lambda_i$, we define
		$$\ell_f(i+1)=\ell_f(i), \quad \ell_g(i+1)=\ell_g(i); \quad r_f(i+1)=r_f(i)-\ell_f(i), \quad r_g(i+1)=r_g(i)-\ell_g(i).$$
	\end{enumerate}
	
	\begin{lemma}\label{lem:ellr}
		For every $i\in\N$ for which they are defined, 
		$$\ell_f(i)>0, \quad \ell_g(i)\leq 0; \quad r_f(i)\leq 0, \quad r_g(i)> 0.$$
	\end{lemma}
	\begin{proof}
		We prove this claim inductively. For $i=1$ it is true by the definition. Assume the claim is true for some $i\geq 1$, and assume $\lambda_{i+1}$ and $\rho_{i+1}$ are well-defined. 
		\begin{enumerate}
			\item Assume $\lambda_i\succ\rho_i$. Then 
			$$\ell_f(i+1)=\ell_f(i)-r_f(i)>0,$$
			since $\ell_f(i)>0$, and $r_f(i)\leq 0$. Also,
			$$\ell_g(i+1)=\ell_g(i)-r_g(i)< 0,$$
			since $\ell_g(i)\leq 0$, and $r_g(i)>0$.
			Furthermore,
			$$r_f(i+1)=r_f(i)\leq 0, \quad \textrm{and } \quad r_g(i+1)=r_g(i)>0.$$
			\item When $\rho_i\succ\lambda_i$ we similarly get 
			$$\ell_f(i+1)>0, \quad \ell_g(i+1)\leq 0; \quad r_f(i+1)\leq 0, \quad r_g(i+1)> 0,$$
			which completes the proof.
		\end{enumerate}
	\end{proof}
	
	\begin{lemma}\label{lem:diff1}
		For each $i\in\N$ for which they are defined,
		$$\ell_f(i)r_g(i)-\ell_g(i)r_f(i)=1.$$
	\end{lemma}
	\begin{proof}
		We proceed inductively. The claim is obviously true for $i=1$. Assume it is true for some $i\geq 1$, and assume $\lambda_{i+1}, \rho_{i+1}$ are well defined. 
		\begin{enumerate}
			\item Assume $\lambda_i\succ\rho_i$. Then
			$$\ell_f(i+1)r_g(i+1)-\ell_g(i+1)r_f(i+1)=(\ell_f(i)-r_f(i))r_g(i)-(\ell_g(i)-r_g(i))r_f(i)=1.$$
			\item Assume $\rho_i\succ\lambda_i$. Then
			$$\ell_f(i+1)r_g(i+1)-\ell_g(i+1)r_f(i+1)=\ell_f(i)(r_g(i)-\ell_g(i))-\ell_g(i)(r_f(i)-\ell_f(i))=1.$$
		\end{enumerate}
	\end{proof}

	\begin{lemma}\label{lem:lambdarho}
		Assume that $f,g\colon X\to X$ are onto non-homeomorphisms which commute. If $i\in\N$ is such that $\lambda_i$ and $\rho_i$ are well-defined, then
		$$\lambda_i=f^{\ell_f(i)}\circ g^{\ell_g(i)}, \quad \rho_i=g^{r_g(i)}\circ f^{r_f(i)}.$$
	\end{lemma}
	\begin{proof}
		The claim is obviously true for $i=1$. Assume it is true for some $i\geq 1$, and assume that $\lambda_{i+1}, \rho_{i+1}$ are well-defined, so $\lambda_i\succ \rho_i$, or $\rho_i\succ\lambda_i$.
		\begin{enumerate}
			\item Assume $\lambda_i\succ \rho_i$. Then
			$$
			{\begin{split}
					\lambda_{i+1} & =\lambda_i\circ\rho_i^{-1}=f^{\ell_f(i)}\circ g^{\ell_g(i)}\circ f^{-r_f(i)}\circ g^{-r_g(i)}=(\textrm{Lemma~\ref{lem:composing}, $f,g$ commute})\\
					& =f^{\ell_f(i)-r_f(i)}\circ g^{\ell_g(i)-r_g(i)}=f^{\ell_f(i+1)}\circ g^{\ell_g(i+1)},
			\end{split}}$$
			and
			$$
			{\begin{split}
					\rho_{i+1} =\rho_i=g^{r_g(i)}\circ f^{r_f(i)}=g^{r_g(i+1)}\circ f^{r_f(i+1)}.
			\end{split}}$$
			\item We proceed similarly when $\rho_i\succ \lambda_i$. 
		\end{enumerate}
	\end{proof}

	\begin{theorem}\label{thm:cousins}
		Assume that $f,g\colon X\to X$ are onto non-homeomorphisms such that $f\neq g$. Then $f$ and $g$ are kin if and only if $f$ and $g$ commute and there is $N\in\N$ such that $\lambda_i, \rho_i$ are well-defined for all $1\leq i\leq N$, and $\lambda_N\circ(\rho_N)^{-1}$ is a homeomorphism. \\
		Maps $f$ and $g$ are cousins if and only if $f$ and $g$ commute and there is $N\in\N$ such that $\lambda_i, \rho_i$ are well-defined for all $1\leq i\leq N$, and $\lambda_N=\rho_N$. 
	\end{theorem}
	\begin{proof}
		
		Assume first that $f, g$ are kin (respectively cousins), so there are $a,b\geq 0, n,m\in\N$, map $\varphi\colon X\to X$ and a homeomorphism $h\colon X\to X$ (if we assume $f,g$ are cousins, we take $h=id$) which commutes with $\varphi$, such that $f=h^a\circ\varphi^n, g=h^b\circ\varphi^m$. Note that, since $h$ and $\varphi$ commute, so do $f$ and $g$.
		
		If $n=m$, then $f\circ g^{-1}=\lambda_1\circ\rho_1^{-1}=h^{a-b}$, so we are done.
		
		Assume without loss of generality that $n>m$. Let $n_1, \ldots, n_N\in\N$ and $m_1, \ldots, m_N\in\N$ be positive integers obtained from the Euclidean algorithm on $n, m$. By Lemma~\ref{lem:hnhm2}, it is easy to see that there are $a_i, b_i\in\Z$ such that $\lambda_i=h^{a_i}\circ \varphi^{n_i}, \rho_i=h^{b_i}\circ \varphi^{m_i}$ for every $i\in\{1, \ldots, N\}$, and $\lambda_N=h^{a_N}\circ \varphi^k$, $\rho_N=h^{b_N}\circ\varphi^k$, where $k=gcd(n,m)$. Thus, $\lambda_N\circ\rho_N^{-1}=h^{a_N-b_N}$, which is a homeomorphism. (If $h=id$, then $\lambda_N=\rho_N$).
		
		For the other direction, assume that $f,g\colon X\to X$ are onto non-homeomorphisms which commute and there is $N\in\N$ such that $\lambda_i, \rho_i$ are well-defined for all $1\leq i\leq N$, and $\lambda_N\circ\rho_N^{-1}$ is a homeomorphism, denote it by $h'$. (Respectively, $h'=id$).
		
		Lemma~\ref{lem:lambdarho} implies that $\lambda_N=f^{\ell_f(N)}\circ g^{\ell_g(N)}$, and $\rho_N=g^{r_g(N)}\circ f^{r_f(N)}$. By Lemma~\ref{lem:ellr}, we know that $\ell_f(N), r_g(N)> 0$, and $\ell_g(N),r_f(N)\leq 0$. For simplicity we will write $\ell_f=\ell_f(N), \ell_g=\ell_g(N), r_f=r_f(N), r_g=r_g(N)$ in the rest of the proof.
		
		Since we assumed that $\lambda_N\circ\rho_N^{-1}=h'$, we get
		$$f^{\ell_f}\circ g^{\ell_g}=h'\circ g^{r_g}\circ f^{r_f}.$$ 
		Since $f$ and $g$ commute, by composing first with $g^{-\ell_g}$ on the right, using Lemma~\ref{lem:composing}, and then composing with $f^{-r_f}$ on the right, we get
		$$f^{\ell_f-r_f}=h'\circ g^{r_g-\ell_g}.$$
		Denote by $m=\ell_f-r_f$, and $n=r_g-\ell_g$, so $f^m=h'\circ g^n$. Note that by Lemma~\ref{lem:composing}, $h'$ commutes with both $f$ and $g$.
		
		By Lemma~\ref{lem:diff1}, we have $n\ell_f+m\ell_g=r_g\ell_f-\ell_g\ell_f+\ell_f\ell_g-r_f\ell_g=r_g\ell_f-r_f\ell_g=1$. Thus, $n$ and $m$ are relatively prime and we conclude that $f$ and $g$ are kin from Theorem~\ref{thm:kin}. (Respectively, if $h'=id$, we conclude that $f$ and $g$ are cousins using second part of Theorem~\ref{thm:kin}). 
	\end{proof}
	
	\begin{example}\label{ex:big}
		Assume maps $f,g\colon I\to I$ are given by their graphs in Figure~\ref{fig:big}.
		
		We observe that $f\circ g^{-1}$ is a map, and $g\circ f^{-1}$ is not a map, see Figure~\ref{fig:big}. Thus, $f\succ g$, and we set $\lambda_1=f$, $\rho_1=g$, and $\lambda_2=\lambda_1\circ \rho_1^{-1}$, $\rho_2=\rho_1$. 
		We further note that $\rho_2\circ \lambda_2^{-1}$ is a map, but $\lambda_2\circ \rho_2^{-1}$ is not (again see Figure~\ref{fig:big}). Thus, $\rho_2\succ \lambda_2$, and we set $\lambda_3=\lambda_2$, and $\rho_3=\rho_2\circ \lambda_2^{-1}$.
		We conclude similarly that $\lambda_3\succ\rho_3$, and define $\lambda_4=\lambda_3\circ\rho_3^{-1}$, $\rho_4=\rho_3$ (see Figure~\ref{fig:big}). Finally, since $\lambda_4=\rho_4$, we conclude that $f$ and $g$ are cousins, and one of their ancestors is $\lambda_4=\rho_4$. 
	\end{example}

	\begin{figure}
		\includegraphics[scale=1.5]{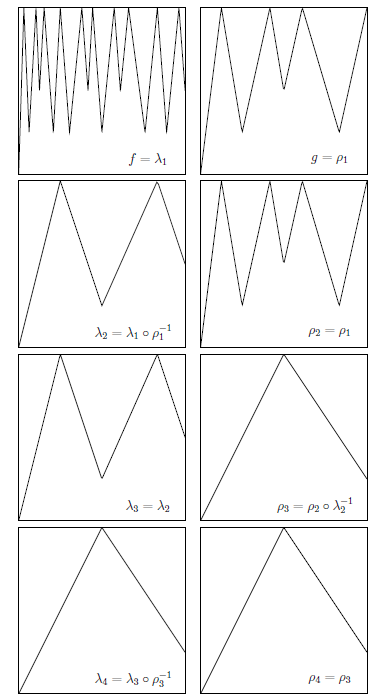}
		\caption{Steps in the algorithm applied to functions $f,g$ from Example~\ref{ex:big}. Since $\lambda_4=\rho_4$, we conclude that $f$ and $g$ are cousins, and $\lambda_4(=\rho_4)$ is one of their ancestors.}
		\label{fig:big}
	\end{figure}
	
	\begin{example}
		Let $\varphi\colon X\to X$ be an onto map which is not a homeomorphism, and let $f=\varphi^{30}, g=\varphi^8$. Obviously, $f$ and $g$ are cousins. The algorithm proceeds as follows (compare to Euclidean algorithm for numbers 30 and 8).
		\begin{center}
			\begin{tabular}{|c|c|c|c|c|c|c|c|} 
				\hline
				$i$ & $\lambda_i$ & $\rho_i$ & $\ell_f(i)$ & $\ell_g(i)$ & $r_g(i)$ & $r_f(i)$ \\ [0.5ex] 
				\hline
				1 & $\varphi^{30}=f^1g^0$ & $\varphi^8=g^1f^0$ & 1 & 0 & 1 & 0 \\ 
				\hline
				2 & $\varphi^{22}=f^1g^{-1}$ & $\varphi^8=g^1f^0$ & 1 & -1 & 1 & 0 \\ 
				\hline
				3 & $\varphi^{14}=f^1g^{-2}$ & $\varphi^8=g^1f^0$ & 1 & -2 & 1 & 0 \\ 
				\hline
				4 & $\varphi^6=f^1g^{-3}$ & $\varphi^8=g^1f^0$ & 1 & -3 & 1 & 0 \\ 
				\hline
				5 & $\varphi^6=f^1g^{-3}$ & $\varphi^2=g^4f^{-1}$ & 1 & -3 & 4 & -1 \\
				\hline
				6 & $\varphi^4=f^2g^{-7}$ & $\varphi^2=g^4f^{-1}$ & 2 & -7 & 4 & -1 \\
				\hline
				7 & $\varphi^2=f^3g^{-11}$ & $\varphi^2=g^4f^{-1}$ & 3 & -11 & 4 & -1 \\
				\hline
			\end{tabular}
		\end{center}
		We note that $\lambda_7=\rho_7=\varphi^2$, so $\varphi^2$ is a common ancestor of $f$ and $g$, and for $m=\ell_f(7)-r_f(7)=4$, and $n=r_g(7)-\ell_g(7)=15$, we get $(\varphi^2)^n=f$, and $(\varphi^2)^m=g$.
	\end{example}
	
	\begin{remark}
		If $f,g$ are not homeomorphisms, and there is a map $\varphi$ such that $f=\varphi^n$, $g=\varphi^m$, $n>m$, according to Lemma~\ref{lem:terminate}, this algorithm will end in at most $n$ steps (it will take the longest in the case when $m=1$ or $m=n-1$). Thus, if the algorithm does not end in $N$ steps, we can conclude that if $f,g$ are cousins, then for every ancestor $\varphi$ of $f$ and $g$, $f=\varphi^n$, $g=\varphi^m$ can only happen when at least one of $n,m > N$. 
	\end{remark}
	
	\begin{example}\label{ex:tent}
		Let $f,g\colon I\to I$, $f=T_2$ and $g=4x(1-x)$. Then $f^{-1}\circ f=g^{-1}\circ g$, so there is a homeomorphism $h\colon I\to I$ such that $f=h\circ g$. However, $f$ and $g$ do not commute, so they are not kin. See Figure~\ref{fig:logistic}.
	\end{example}
	
	\begin{figure}[ht!]
		\centering
		\begin{tikzpicture}[scale=4]
			\draw (0,0)--(0,1)--(1,1)--(1,0)--(0,0);
			\draw[blue, thick] (0,0)--(1/2,1)--(1,0);
			\node[above] at (1/2,1) {\small \blue{$f$} and \red{$g$}};
			\draw[domain=0:1, smooth, variable=\x, red, thick]
			plot ({\x}, {4*\x*(1-\x)});
		\end{tikzpicture}
		\hspace{5pt}
		\begin{tikzpicture}[scale=4]
			\draw (0,0)--(0,1)--(1,1)--(1,0)--(0,0);
			\draw[thick] (0,0)--(1,1);
			\draw[thick] (0,1)--(1,0);
			\node[above] at (1/2,1) {\small $f^{-1}\circ f=g^{-1}\circ g$};
		\end{tikzpicture}
		\hspace{5pt}
		\begin{tikzpicture}[scale=4]
			\draw (0,0)--(0,1)--(1,1)--(1,0)--(0,0);
			\draw[domain=0:0.1464, smooth, variable=\x, blue, thick]
			plot ({\x}, {8*\x*(1-\x)});
			\draw[domain=0.1464:0.8536, smooth, variable=\x, blue, thick]
			plot ({\x}, {2-8*\x*(1-\x)});
			\draw[domain=0.8536:1, smooth, variable=\x, blue, thick]
			plot ({\x}, {8*\x*(1-\x)});
			\draw[domain=0:0.5, smooth, variable=\x, red, thick]
			plot ({\x}, {8*\x*(1-2*\x)});
			\draw[domain=0.5:1, smooth, variable=\x, red, thick]
			plot ({\x}, {8*(1-\x)*(2*\x-1)});
			\node[above] at (1/2,1) {\small \blue{$f\circ g$} and \red{$g\circ f$}};
		\end{tikzpicture}
		\caption{Maps $f,g\colon I\to I$ from Example~\ref{ex:tent}.}
		\label{fig:logistic}
	\end{figure}
	
	\begin{example}
		It can happen that the algorithm never terminates. For example, let $f,g\colon I\to I$ be onto non-homeomorphisms given by Figure~\ref{fig:noterm} below. It is easy to check that $f\circ g^{-n}=f$ for all $n\in\N$, so $f\succ g^n$ for all $n\in\N$. Thus, $\lambda_n=f$, and $\rho_n=g$ for all $n\in\N$. Note that $f$ and $g$ do not commute.
		\begin{figure}[ht!]
			\begin{tikzpicture}[scale=5]
				\draw (0,0)--(0,1)--(1,1)--(1,0)--(0,0);
				\draw[thick] (0,0)--(1/9,1)--(2/9,0)--(1/3,1/2)--(2/3,1/2)--(7/9,1)--(8/9,0)--(1,1);
				\draw[thin, dashed] (1/3,0)--(1/3,1);
				\draw[thin, dashed] (2/3,0)--(2/3,1);
				\node[below] at (1/2,1) {$f$};
			\end{tikzpicture}
			\hspace{5pt}
			\begin{tikzpicture}[scale=5]
				\draw (0,0)--(0,1)--(1,1)--(1,0)--(0,0);
				\draw[thick] (0,0)--(1/3,1/3)--(4/9,2/3)--(5/9,1/3)--(2/3,2/3)--(1,1);
				\draw[thin, dashed] (1/3,0)--(1/3,1);
				\draw[thin, dashed] (2/3,0)--(2/3,1);
				\draw[thin, dashed] (0,1/3)--(1,1/3);
				\draw[thin, dashed] (0,2/3)--(1,2/3);
				\node[below] at (1/2,1) {$g$};
			\end{tikzpicture}
			\caption{Maps $f,g\colon I\to I$ such that $f\succ g^n$ for all $n\in\N$.}
			\label{fig:noterm}
		\end{figure}
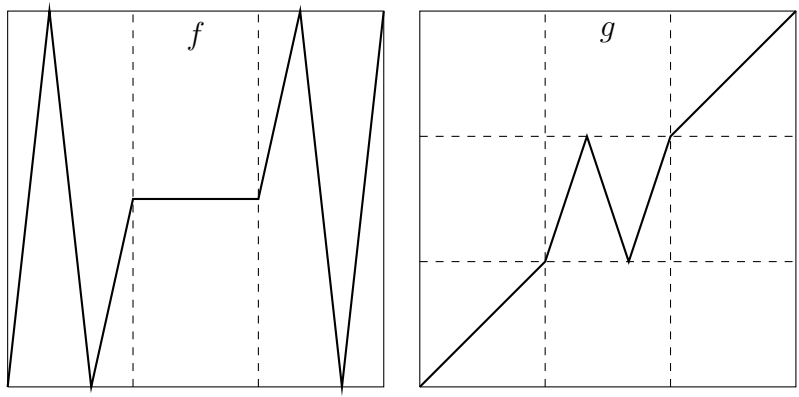
	\end{example}
	
	Below we provide another example of non-homeomorphisms $f,g\colon X\to X$ for which the algorithm never terminates, and which are nowhere-constant.

	\begin{example}\label{ex:fg}
		We will construct onto maps $f,g\colon I\to I$ such that $g$ is piecewise monotone, $f$ and $g$ commute, and for all $n\in\N$, $f\circ g^{-n}$ is a map which is not a homeomorphism.
		
		Let $0\leq a<b\leq 1$, and $0\leq c<d\leq 1$. We define a map $L_{[a,b]}^{[c,d]}\colon [a,b]\to [c,d]$ as
		$$L_{[a,b]}^{[c,d]}(x):=\frac{d-c}{b-a}(x-a)+c,$$
		for all $x\in[a,b]$. Note that $L_{[a,b]}^{[c,d]}$ is simply the orientation-preserving linear map from $[a,b]$ to $[c,d]$, and
		$$\left(L_{[a,b]}^{[c,d]}\right)^{-1}=L_{[c,d]}^{[a,b]}.$$
		
		Given an onto map $\psi\colon I\to I$, and $0\leq a<b\leq 1$, and $0\leq c<d\leq 1$, we define $\psi_{[a,b]}^{[c,d]}\colon [a,b]\to [c,d]$ as
		$$\psi_{[a,b]}^{[c,d]}:=L_I^{[c,d]}\circ \psi\circ L_{[a,b]}^{I}.$$
		
		\begin{claim}\label{cl:phis}
			Given $0\leq a<b\leq 1$, $0\leq c<d\leq 1$, and $0\leq c'<d'\leq 1$,
			$$\psi_{[a,b]}^{[c,d]}\circ \left(\psi_{[a,b]}^{[c',d']}\right)^{-1}=L_{[c',d']}^{[c,d]}.$$
		\end{claim}
		\begin{proof}[Proof of Claim~\ref{cl:phis}]
			We have
			\begin{equation*}
				\begin{split}
					\psi_{[a,b]}^{[c,d]}\circ \left(\psi_{[a,b]}^{[c',d']}\right)^{-1} & =L_I^{[c,d]}\circ\psi\circ L_{[a,b]}^I\circ \left(L_I^{[c',d']}\circ\psi\circ L_{[a,b]}^I\right)^{-1}\\
					&=L_I^{[c,d]}\circ\psi\circ L_{[a,b]}^I\circ L_I^{[a,b]}\circ\psi^{-1}\circ L_{[c',d']}^I\\
					&=L_I^{[c,d]}\circ id_I \circ L_{[c',d']}^I=L_{[c',d']}^{[c,d]}.
				\end{split}
			\end{equation*}
		\end{proof}
		
		We take and piecewise monotone onto map $\psi\colon I\to I$, which is not a homeomorphism, such that $\psi(0)=0, \psi(1)=1$, and
		$$0<\ldots<a_{-2}<b_{-2}<a_{-1}<b_{-1}<a_0<b_0<a_1<b_1<a_2<b_2<\ldots<1,$$
		such that $a_{-i}\to 0$, and $a_i\to 1$ as $i\to\infty$.
		For every $k\in\Z$ we let	
		$$I_k:=[a_k,b_k], \quad \text{and } J_k:=[b_k, a_{k+1}].$$
		Now we define $f,g\colon I\to I$,
		$$f(x):=\begin{cases}
			0, & x=0,\\
			L_{J_k}^{J_{k+1}}(x), & x\in J_k,\\
			\psi_{I_k}^{I_{k+1}}(x), & x\in I_k, k\neq -1,\\
			L_{I_{-1}}^{I_0}(x), & x\in I_{-1},\\
			1, & x=1.
		\end{cases}$$
		$$g(x):=\begin{cases}
			0, & x=0,\\
			L_{J_k}^{J_{k-1}}(x), & x\in J_k,\\
			L_{I_k}^{I_{k-1}}(x), & x\in I_k, k\neq 0,\\
			\psi_{I_0}^{I_{-1}}(x), & x\in I_0,\\
			1, & x=1.
		\end{cases}$$
		Note that $f|_{I_k}\colon I_k\to I_{k+1}, f|_{J_k}\colon J_k\to J_{k+1}$, and $g|_{I_k}\colon I_k\to I_{k-1}, g|_{J_k}\colon J_k\to J_{k-1}$, for all $k\in\Z$. In particular, $f(a_k)=a_{k+1}, f(b_k)=b_{k+1}, g(a_k)=a_{k-1}, g(b_k)=b_{k-1}$ for all $k\in\Z$, and $f(0)=g(0)=0, f(1)=g(1)=1$, so $f$ and $g$ are continuous and onto. Note that $g$ is piecewise monotone. With a careful choice of $a_k, b_k$, $k\in\Z$, and $\psi\colon I\to I$, we could make $g$ piecewise linear if desired. For example, see Figure~\ref{fig:fg}.

		\begin{figure}[ht!]
			\begin{tikzpicture}[scale=8]
				\draw (0,0)--(0,1)--(1,1)--(1,0)--(0,0);
				\draw[dashed, thin] (0,0)--(1,1);
				
				
				\newcommand*\List{{2/256, 2/128, 2/64, 2/32, 2/16, 2/8, 3/8, 5/8, 6/8, 14/16, 30/32, 62/64, 126/128, 254/256}}
				\foreach \i in {0,...,10} {
					\draw[dashed] (\List[\i],\List[\i+1]) rectangle (\List[\i+1],\List[ \i+2]);
				}
				\foreach \i in {2,...,12} {
					\draw[dashed] (\List[\i],\List[\i-1]) rectangle (\List[\i-1],\List[ \i-2]);
				}
				\foreach \i in {2,...,6} {
					\draw[thick, blue] (\List[\i],\List[\i-1])--(\List[\i-1],\List[ \i-2]);
				}
				\foreach \i in {6}{
					\pgfmathparse{\List[\i]}
					\let\b\pgfmathresult		
					\pgfmathparse{\List[\i+1]}
					\let\c\pgfmathresult
					\pgfmathparse{(\c - \b)/3}
					\let\r\pgfmathresult
					\draw[thick, blue] (\List[\i], \List[\i-1])--(\List[\i]+\r, \List[\i])--(\List[\i]+\r+\r, \List[\i-1])--(\List[\i+1], \List[\i]);
				}
				\foreach \i in {8,...,12} {
					\draw[thick, blue] (\List[\i],\List[\i-1])--(\List[\i-1],\List[ \i-2]);
				}		
				
				\foreach \i in {0,...,4}{
					\pgfmathparse{\List[\i]}
					\let\b\pgfmathresult		
					\pgfmathparse{\List[\i+1]}
					\let\c\pgfmathresult
					\pgfmathparse{(\c - \b)/3}
					\let\r\pgfmathresult
					\draw[thick, red] (\List[\i], \List[\i+1])--(\List[\i]+\r, \List[\i+2])--(\List[\i]+\r+\r, \List[\i+1])--(\List[\i+1], \List[\i+2]);
				}
				\foreach \i in {5}{
					\draw[thick, red] (\List[\i],\List[\i+1])--(\List[\i+1],\List[\i+2]);
				}
				\foreach \i in {6,...,11}{
					\pgfmathparse{\List[\i]}
					\let\b\pgfmathresult		
					\pgfmathparse{\List[\i+1]}
					\let\c\pgfmathresult
					\pgfmathparse{(\c - \b)/3}
					\let\r\pgfmathresult
					\draw[thick, red] (\List[\i], \List[\i+1])--(\List[\i]+\r, \List[\i+2])--(\List[\i]+\r+\r, \List[\i+1])--(\List[\i+1], \List[\i+2]);
				}
			\end{tikzpicture}
			\caption{Maps $f\colon I\to I$ (in red) and $g\colon I\to I$ (in blue) which commute and for which $f\circ g^{-n}$ is a map which is not a homeomorphism for each $n\in\N$ as in Example~\ref{ex:fg}. Here $0<\ldots a_{-2}<b_{-2}=a_{-1}<b_{-1}=a_0<b_0=a_1<b_1=a_2<\ldots 1$, and map $\psi\colon I\to I$ is a symmetric 3-tent map.}
			\label{fig:fg}
		\end{figure}
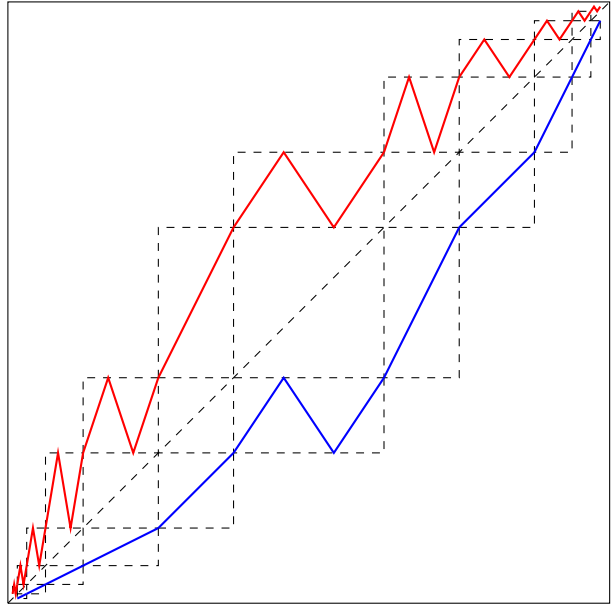

		\begin{claim}\label{cl:fgcomm}
			$f$ and $g$ commute.
		\end{claim}
		\begin{proof}[Proof of Claim~\ref{cl:fgcomm}]
			Let $k\in\Z$, and $x\in J_k$. We have
			$$f\circ g(x)=f(L_{J_k}^{J_{k-1}}(x))=L_{J_{k-1}}^{J_k}(L_{J_k}^{J_{k-1}}(x))=x=L_{J_{k+1}}^{J_k}(L_{J_k}^{J_{k+1}}(x))=g(L_{J_k}^{J_{k+1}}(x))=g\circ f(x).$$
			
			Let $k\in\Z$, $k\neq 0,1$, and $x\in I_k$. We have
			$$f\circ g(x)=f(L_{I_k}^{I_{k-1}}(x))=\psi_{I_{k-1}}^{I_k}\circ L_{I_k}^{I_{k-1}}(x)=L_I^{I_k}\circ\psi\circ L_{I_{k-1}}^{I}\circ L_{I_k}^{I_{k-1}}(x)=L_I^{I_k}\circ\psi\circ L_{I_k}^{I}(x),$$
			$$g\circ f(x)=g(\psi_{I_k}^{I_{k+1}}(x))=L_{I_{k+1}}^{I_k}\circ\psi_{I_k}^{I_{k+1}}(x)=L_{I_{k+1}}^{I_k}\circ L_I^{I_{k+1}}\circ\psi\circ L_{I_k}^I(x)=L_I^{I_k}\circ\psi\circ L_{I_k}^{I}(x),$$
			so $f\circ g(x)=g\circ f(x)$.
			
			Let $x\in I_0$. Then
			$$f\circ g(x)=f(\psi_{I_0}^{I_{-1}}(x))=L_{I_{-1}}^{I_0}\circ L_I^{I_{-1}}\circ\psi\circ L_{I_0}^I(x)=L_I^{I_0}\circ\psi\circ L_{I_0}^I(x),$$
			$$g\circ f(x)=g(\psi_{I_0}^{I_1}(x))=L_{I_1}^{I_0}\circ L_I^{I_1}\circ \psi\circ L_{I_0}^I(x)=L_I^{I_0}\circ\psi\circ L_{I_0}^I(x),$$
			so $f\circ g(x)=g\circ f(x)$.
			
			Let $x\in I_{-1}$. Then
			$$f\circ g(x)=f(L_{I_{-1}}^{I_{-2}}(x))=\psi_{I_{-2}}^{I_{-1}}\circ L_{I_{-1}}^{I_{-2}}(x)=L_{I}^{I_{-1}}\circ\psi\circ L_{I_{-2}}^I\circ L_{I_{-1}}^{I_{-2}}(x)=L_I^{I_{-1}}\circ\psi\circ L_{I_{-1}}^I(x),$$
			$$g\circ f(x)=g(L_{I_{-1}}^{I_0}(x))=\psi_{I_0}^{I_{-1}}\circ L_{I_{-1}}^{I_0}(x)=L_I^{I_{-1}}\circ\psi\circ L_{I_0}^I\circ L_{I_{-1}}^{I_0}(x)=L_I^{I_{-1}}\circ\psi\circ L_{I_{-1}}^I(x),$$
			so $f\circ g(x)=g\circ f(x)$.
			
			Note that $g\circ f|_{I_k}=f\circ g|_{I_k}=\psi_{I_k}^{I_k}$, and $g\circ f|_{J_k}=f\circ g|_{J_k}=id_{J_k}$ for all $k\in\Z$.
		\end{proof}
		
		\begin{claim}\label{cl:itg}
			For every $n\in\N$, 
			$$g^n|_{I_k}=\begin{cases}
				\psi_{I_k}^{I_{k-n}}, & k\in\{0,1,\ldots, n-1\},\\
				L_{I_k}^{I_{k-n}}, & k\in\Z\setminus\{0,1,\ldots,n-1\}.
			\end{cases},$$
			and
			$$g^n|_{J_k}=L_{J_k}^{J_{k-n}},$$
			for all $k\in\Z$.
		\end{claim}
		\begin{proof}[Proof of Claim~\ref{cl:itg}]
			Note that $g^n|{I_k}=g|_{I_{k-n+1}}\circ\ldots\circ g|_{I_{k-1}}\circ g|_{I_k}$. Since $g|_{I_l}=L_{I_l}^{I_{l-1}}$ whenever $l\neq 0$, we immediately get $g^n|_{I_k}=L_{I_k}^{I_{k-n}}$ if $0\not\in\{k, k-1, \ldots, k-n+1\}$, \ie $k\not\in\{0, 1, \ldots, n-1\}$. We similarly get $g^n|_{J_k}=L_{J_k}^{J_{k-n}}$ for all $k\in\Z$.
			
			Assume that $k\in\{0,1,\ldots, n-1\}$, \ie $0\in\{k, k-1, \ldots, k-n+1\}$. Then $g|_{I_{k-l}}=\psi_{I_0}^{I_{-1}}$ for some $l\in\{0,1,\ldots, n-1\}$, and
			\begin{equation*}
				\begin{split}
					g^n|_{I_k} & =L_{I_{k-n+1}}^{I_{k-n}}\circ\ldots\circ L_{I_{k-l-1}}^{I_{k-l-2}}\circ\psi_{I_{k-l}}^{I_{k-l-1}}\circ L_{I_{k-l+1}}^{I_{k-l}}\circ\ldots\circ L_{I_k}^{I_{k-1}}\\
					& = L_{I_{k-l-1}}^{I_{k-n}}\circ L_{I}^{I_{k-l-1}}\circ\psi\circ L_{I_{k-l}}^{I}\circ L_{I_k}^{I_{k-l}} = L_{I}^{I_{k-n}}\circ\psi\circ L_{I_{k}}^{I}\\
					& = \psi_{I_k}^{I_{k-n}}.
				\end{split}
			\end{equation*}
		\end{proof}
		
		\begin{claim}\label{cl:neverending}
			$f\circ g^{-n}$ is a map which is not a homeomorphism for every $n\in\N$. 
		\end{claim}
		\begin{proof}[Proof of Claim~\ref{cl:neverending}]
			Let $n\in\N$. By Claim~\ref{cl:itg}, since $g^n|_{J_k}$ is a homeomorphism for all $k\in\Z$, and $g^n|_{I_k}$ is a homeomorphism for all $k\in\Z\setminus\{0,1,\ldots,n-1\}$, we only need to check that $f\circ (g^n|_{I_k})^{-1}$ is a map for $k\in\{0,1,\ldots,n-1\}$. Note that if $n\in\N$ and $k\in\{0,1,\ldots,n-1\}$, then $k\neq -1$, so $f|_{I_k}=\psi_{I_k}^{I_{k+1}}$ (since $k\neq -1$), and $g^n|_{I_k}\psi_{I_k}^{I_{k-n}}$ by Claim~\ref{cl:itg}. Thus,
			$$f\circ (g^n|_{I_k})^{-1}=\psi_{I_k}^{I_{k+1}}\circ (\psi_{I_k}^{I_{k-n}})^{-1}=L_{I_{k-n}}^{I_{k+1}},$$
			where the last equality follows from Claim~\ref{cl:phis}. 
			
			To prove that $f\circ g^{-n}$ is not a homeomorphism for every $n\in\N$, note, for example, that $g^n|_{I_n}=L_{I_n}^{I_0}$, so $f\circ g^{-n}|_{I_0}=f\circ L_{I_0}^{I_n}=\psi_{I_n}^{I_{n+1}}\circ L_{I_0}^{I_n}$, which is not a homeomorphism, since $\psi$ is a non-homeomorphism.
		\end{proof}
		
		Thus, $\lambda_n=f\circ g^{-n+1}$, and $\rho_n=g$ for all $n\in\N$ in the algorithm above, so it never terminates. 
	\end{example}

	\begin{remark}
		Let $f,g\colon I\to I$ be onto non-homeomorphisms which are piecewise monotone. Assume that $n\in\N$ is such that $\lambda_i, \rho_i$ are well-defined for all $1\leq i\leq n$. Then $\lambda_{n+1}, \rho_{n+1}$ is well-defined if and only if exactly one of $\lambda_n\circ\rho_n^{-1}$ or $\rho_n\circ\lambda_n^{-1}$ is a map (which is not a homeomorphism). Assume $\lambda_n\succ \rho_n$, so $\lambda_n\circ\rho_n^{-1}$ is a map which is not a homeomorphism. Then the number of critical points of $\lambda_{n+1}=\lambda_n\circ\rho_n^{-1}$ is strictly less than the number of critical points of $\lambda_n$. Since $f$ and $g$ have finitely many critical points, the algorithm will terminate in a finite number of steps. Actually, of $n_f$ is the number of critical points of $f$, and $n_g$ is the number of critical points of $g$, the algorithm will terminate in at most $\max\{n_f,n_g\}$ steps.
	\end{remark}
	
	\begin{remark}
		Recall that in order to algorithmically check whether maps $f,g\colon X\to X$ are kin using Theorem~\ref{thm:cousins}, we needed to assume that $f$ and $g$ are {\bf non-homeomorphisms}. In light of Remark~\ref{rem:homeo},  note that homeomorphisms $f$ and $g$ are kin if and only if they commute. However, determining when two homeomorphisms are cousins is less trivial. For example, let $f, g\colon I\to I$ be such that $f=id_I$, and $g\neq id_I$ is any orientation-preserving homeomorphism. Then $f$ and $g$ commute (so they are kin), but $f^{\alpha}\neq g^{\beta}$ for all $\alpha, \beta\in\N$, so they are not cousins according to Theorem~\ref{thm:kin}. If $f$ and $g$ are homeomorphisms, to design an algorithm to check if they are cousins, we need a different and more appropriate definition of the relation $f\succ g$. Such a relation would need to have the property that for every homeomorphism $\varphi\colon X\to X$, $\varphi^n\succ \varphi^m$ if and only if $n>m$. We provide such a relation in case $X=I$ in the following subsection. It is not clear how to define such a relation on more general space $X$, especially if the dimension of $X$ is greater than 1.  
	\end{remark}
	
	\subsection{Euclidean algorithm on homeomorphisms}\label{subsec:homeo}

	\begin{definition}
		Let $f,g\colon I\to I$ be (not necessarily increasing) homeomorphisms. We write $f\succ g$ if $f^2\neq g^2$ and for every $x\in I$ either:
		\begin{enumerate}
			\item $x\leq g^2(x)\leq f^2(x)$ or
			\item $f^2(x)\leq g^2(x)\leq x$.
		\end{enumerate}
	\end{definition}
	
	\begin{lemma}\label{lem:incrhomeos2}
		Assume $h\colon I\to I$ is a homeomorphism such that $h^2\neq id$. Then $h^n\succ h^m$ if and only if $n>m$.
	\end{lemma}
	\begin{proof} Since $h^2$ is increasing, for every $x\in I$ either 
		$$x\leq h^2(x)\leq h^4(x)\leq h^6(x)\leq\ldots, \textrm{ or}$$
		$$x\geq h^2(x)\geq h^4(x)\geq h^6(x)\geq \ldots.$$
		Moreover, since $h^2\neq id$, there is $a\in I$ such that $h^2(a)\neq a$, so either 
		$$a< h^2(a)< h^4(a)< h^6(a)<\ldots, \textrm{ or}$$
		$$a> h^2(a)> h^4(a)> h^6(a)> \ldots.$$
		
		Assume that $n,m\in\N$ are such that $h^n\succ h^m$. Then $h^{2n}\neq h^{2m}$ so in particular $n\neq m$. By definition, for every $x\in I$, either $x\leq h^{2m}(x)\leq h^{2n}(x)$, or $x\geq h^{2m}(x)\geq h^{2n}(x)$. Thus, $a<h^{2m}(a)<h^{2n}(a)$, or $a>h^{2m}(a)>h^{2n}(a)$. According to the relations above, it follows that $m<n$.
		
		Similarly, if $m<n$, then for every $x\in I$, either $x\leq h^{2m}(x)\leq h^{2n}(x)$, or $x\geq h^{2m}(x)\geq h^{2n}(x)$, and, in particular, $a<h^{2m}(a)<h^{2n}(a)$, or $a>h^{2m}(a)>h^{2n}(a)$. Thus, $h^{2m}\neq h^{2n}$ and $h^n\succ h^m$.  
	\end{proof}
	
	\begin{remark}
		Homeomorphisms $h\colon I\to I$ for which $h^2=id$ are either themselves the $id$, or orientation reversing $h\colon I\to I$ which are of the form
		$$h(x)=\begin{cases}
			h_0(x), & x\in[0,c],\\
			h_0^{-1}(x), & x\in[c,1],
		\end{cases}$$
		where $c$ is the unique fixed point of $h$, and $h_0\colon [0,c]\to[c,1]$ is any homemorphism such that $h_0(0)=1$, and $h_0(c)=c$. If $h^2=id$, then, by definition, $h\not\succ id$, and $id\not\succ h$. Note that in this case $h$ and $id$ are cousins (with ancestor $h$), and $h^n=id$ for all even $n\in\N$, and $h^n=h$ for all odd $n\in\N$. 
	\end{remark}
	
	As before, we present an algorithm that decides if two homeomorphisms $f,g\colon I\to I$ are cousins, and finds their ancestor.
	
	Let $f,g\colon I\to I$ be homeomorphisms such that $f\succ g$. We define
	$$\lambda_1=f, \quad \rho_1=g.$$
	For every $i\in\N$ such that $\lambda_i\succ \rho_i$, or $\rho_i\succ \lambda_i$:
	\begin{enumerate}
		\item if $\lambda_i\succ \rho_i$, we define
		$$\lambda_{i+1}=\lambda_{i}\circ\rho^{-1}_{i}, \quad \rho_{i+1}=\rho_{i},$$ 
		\item if $\rho_i\succ\lambda_i$, we define
		$$\lambda_{i+1}=\lambda_{i}, \quad \rho_{i+1}=\rho_{i}\circ\lambda^{-1}_{i}.$$ 
	\end{enumerate}
	
	\begin{theorem}\label{thm:alg_homeos}
		Assume that $f,g\colon I\to I$ are homeomorphisms such that $f\neq g$. Then $f$ and $g$ are cousins if and only if either:
		\begin{enumerate}
			\item $f$ and $g$ commute and there is $N\in\N$ such that $\lambda_i, \rho_i$ are well defined for all $1\leq i\leq N$, and $\lambda_N=\rho_N$, or
			\item there is an orientation reversing homeomorphism $h\colon I\to I$ such that $h^2=id$, and $\{f,g\}=\{id,h\}$.
		\end{enumerate} 
	\end{theorem}
	\begin{proof}
		The proof of the {\em if} direction is the same as the proof of Theorem~\ref{thm:cousins}. Conversely, if $f,g$ are cousins with ancestor $h\colon I\to I$, then either $h^2=id$, in which case $\{f,g\}=\{h,id\}$, or the proof proceeds as in Theorem~\ref{thm:cousins}, but with Lemma~\ref{lem:incrhomeos2} in place of Lemma~\ref{lem:hnhm2}.
	\end{proof}
	
	\begin{remark}
		If $f,g\colon I\to I$ are both increasing homeomorphisms, then $(2)$ in Theorem~\ref{thm:alg_homeos} cannot happen. So increasing homeomorphisms $f,g\colon I\to I$ are cousins if and only if they commute and there is $N\in\N$ such that $\lambda_i, \rho_i$ are well defined for all $1\leq i\leq N$, and $\lambda_N=\rho_N$.
	\end{remark}
	
	\begin{remark}
		If $X\neq I$ is a continuum, it is not as easy to define a relation $\succ$ on homeomorphisms on $X$ which would have the property that for all homeomorphisms $\varphi\colon X\to X$, $\varphi^n\succ\varphi^m$ if and only if $n>m$. If $X=S^1$ is a circle, we might be able to lift homeomorphisms $f,g\colon S^1\to S^1$ to $\R$ and define a relation $\succ$ as on $I$, and similar construction might be possible for $X$ which are graphs. However, if $X$ is higher-dimensional, \eg a sphere, defining such a relation would be more complicated. We leave it as an open question below. 
	\end{remark}
	
	\begin{question}
		Let $X$ be a continuum. Can we find a relation $\succ$ on homeomorphisms $f,g\colon X\to X$ which has the following property: for all homeomorphisms $\varphi\colon X\to X$ and $n,m\in\N$, $\varphi^n\succ\varphi^m$ if and only if $n>m$.
	\end{question}

	\section{Diagonal kin maps}\label{sec:entropy}
	
	Recall that, given a continuum $X$ and maps $f,g\colon X\to X$, we say that $f$ and $g$ {\em strongly commute} if $g\circ f^{-1}(x)=f^{-1}\circ g(x)$ for all $x\in X$, \cite{AM-entropy}. By Theorem~\ref{thm:str_comm}, if maps $f,g\colon X\to X$ strongly commute, then the entropy of the $g$-diagonal map $\Psi\colon\varprojlim(X,f)\to\varprojlim(X,f)$ defined as $\Psi((x_0,x_1,x_2,\ldots)):=(g(x_1), g(x_{2}), g(x_{3}), \ldots),$ for all $(x_0,x_1,x_2, \ldots)\in\varprojlim(X,f)$, is equal to
	$$\Ent(\Psi)=\Ent(g\circ f^{-1}).$$
	
	\begin{lemma}\label{lem:shiftdiagonal}
		Let $f,g\colon X\to X$ be maps which strongly commute, and let $k\in\N$. We define $\Psi\colon\varprojlim(X,f)\to\varprojlim(X,f)$ as
		$$\Psi((x_0,x_1,x_2,\ldots)):=(g(x_k), g(x_{k+1}), g(x_{k+2}), \ldots).$$
		Then $\Ent(\Psi)=\Ent(g\circ f^{-k})$.
	\end{lemma} 
	\begin{proof}
		Let $H\colon\varprojlim(X,f)\to\varprojlim(X, f^k)$ be a homeomorphism given by
		$$H((x_0,x_1,x_2, \ldots))=(x_0,x_k,x_{2k}, \ldots),$$
		for all $(x_0,x_1,x_2,\ldots)\in\varprojlim(X,f)$, with inverse
		$$H^{-1}((x_0,x_1,x_2, \ldots)):=(x_0,f^{k-1}(x_1), \ldots,f(x_1), x_1, f^{k-1}(x_2), \ldots, f(x_2), x_2, \ldots),$$
		for all $(x_0,x_1,x_2,\ldots)\in\varprojlim(X,f^k)$.
		We then define a map $\tilde\Psi\colon\varprojlim(X,f^k)\to\varprojlim(X,f^k)$ as 
		$$\tilde\Psi:=H\circ \Psi\circ H^{-1}.$$
		Since $\tilde\Psi$ is conjugate to $\Psi$, by Lemma~\ref{lem:conjugate}, we have $\Ent(\tilde\Psi)=\Ent(\Psi)$.
		Note that
		\begin{equation*}
			\begin{split}
				\tilde\Psi((x_0,x_1,x_2,\ldots))&=H(\Psi((x_0,f^{k-1}(x_1), \ldots,f(x_1), x_1, f^{k-1}(x_2), \ldots, f(x_2), x_2, \ldots)))\\
				&=H(g(x_1), g\circ f^{k-1}(x_2), \ldots g\circ f(x_2), g(x_2), \ldots)\\
				&=(g(x_1), g(x_2), \ldots).
			\end{split}
		\end{equation*}
		Since $f$ and $g$ strongly commute, so do $f^k$ and $g$, and thus Theorem~\ref{thm:str_comm} implies that 
		$$\Ent(\Psi)=\Ent(\tilde\Psi)=\Ent(g\circ f^{-k}).$$
	\end{proof}
	
	\begin{theorem}\label{thm:ent}
		Let $f,g\colon X\to X$ be kin, \ie there is a homeomorphism $h\colon X\to X$, an onto map $\varphi\colon X\to X$ which commutes with $h$, and $a,b\geq 0$, $n,m\in\N$ such that $f=h^a\circ \varphi^n$, and $g=h^b\circ \varphi^m$. Let $G\colon\varprojlim(X,f)\to\varprojlim(X,f)$ be the diagonal map defined as
		$$G((x_0,x_1,x_2, \ldots)):=(g(x_1), g(x_2), \ldots).$$
		Then $\Ent(G)=\Ent(h^{b-a}\circ\varphi^{m-n})$.
	\end{theorem}
	\begin{proof}
		{\bf Case~1.} Assume that $m\geq n$, and let $\psi\colon X\to X$ be the map $\psi=h^{b-a}\circ \varphi^{m-n}$. Then, since $h$ and $\varphi$ commute, we have
		$$\psi\circ f=h^{b-a}\circ\varphi^{m-n}\circ h^a\circ \varphi^n=h^b\circ\varphi^m=g.$$
		Thus, for every $(x_0,x_1,\ldots)\in\varprojlim(X,f)$, we have
		$$G((x_0,x_1,x_2,\ldots))=(g(x_1), g(x_2), \ldots)=(\psi\circ f(x_1), \psi\circ f(x_2),\ldots)=(\psi(x_0), \psi(x_1),\ldots),$$
		so Theorem~\ref{thm:Ye} implies $\Ent(G)=\Ent(\psi)=\Ent(h^{b-a}\circ\varphi^{m-n})$. 
		
		{\bf Case~2.} Let $m<n$. 
		
		We first construct a map $\tilde G\colon\varprojlim(X,\varphi)\to\varprojlim(X,\varphi)$ which is conjugate to $G$, and then compute $\Ent(\tilde G)$.
		
		Let $H_1\colon\varprojlim(X,f)\to\varprojlim(X,\varphi^n)$ be a homeomorphism given by 
		$$H_1((x_0,x_1,x_2,x_3,\ldots)):=(x_0,h^a(x_1), h^{2a}(x_2), h^{3a}(x_3),\ldots),$$
		for all $(x_0,x_1,x_2,x_3,\ldots)\in\varprojlim(X,f)$.
		Note that $h^{ia}\circ f=h^{ia}\circ h^{a}\circ\varphi^n=h^{(i+1)a}\circ\varphi^n$ for all $i\geq 0$, since $h$ and $\varphi$ commute. So, for every $i\geq 0$, $\varphi^n(h^{(i+1)a}(x_{i+1}))=h^{ia}(f(x_{i+1}))=h^{ia}(x_i)$, showing that the map $H_1$ is well-defined. The inverse of $H_1$ is
		$$H_1^{-1}((x_0,x_1,x_2,x_3,\ldots))=(x_0,h^{-a}(x_1), h^{-2a}(x_2), h^{-3a}(x_3), \ldots),$$
		for all $(x_0,x_1,x_2,x_3,\ldots)\in\varprojlim(X,\varphi^n)$.
		
		We now define $G_1\colon\varprojlim(X,\varphi^n)\to\varprojlim(X,\varphi^n)$ as
		$$G_1:=H_1\circ G\circ H_1^{-1},$$
		so $G_1$ is conjugate to $G$ and
		\begin{equation*}
			\begin{split}
				G_1((x_0,x_1,x_2,\ldots)) &=H_1(G((x_0,h^{-a}(x_1), h^{-2a}(x_2), h^{-3a}(x_3), \ldots)))\\
				&=H_1((g\circ h^{-a}(x_1), g\circ h^{-2a}(x_2), g\circ h^{-3a}(x_3), \ldots))\\
				&=(g\circ h^{-a}(x_1), g\circ h^{-a}(x_2), g\circ h^{-a}(x_3), \ldots)\\
				&=(h^{b-a}\circ\varphi^m(x_1), h^{b-a}\circ\varphi^m(x_2), h^{b-a}\circ\varphi^m(x_3), \ldots),
			\end{split}
		\end{equation*}
		for all $(x_0,x_1,x_2,x_3,\ldots)\in\varprojlim(X,\varphi^n)$.
		
		Now we define a homeomorphism $H_2\colon\varprojlim(X,\varphi^n)\to\varprojlim(X,\varphi)$ as
		$$H_2((x_0,x_1,x_2,\ldots)):=(x_0,\varphi^{n-1}(x_1), \ldots, \varphi(x_1), x_1, \varphi^{n-1}(x_2), \ldots, \varphi(x_2), x_2, \ldots),$$
		with the inverse
		$$H_2^{-1}((x_0,x_1,x_2,\ldots))=(x_0,x_n,x_{2n},\ldots).$$
		We then define a homeomorphism $\tilde G\colon\varprojlim(X,\varphi)\to\varprojlim(X,\varphi)$ as
		$\tilde G:=H_2\circ G_1\circ H_2^{-1}$, which is thus conjugated to $G_1$, and thus also to $G$. To be more precise, $\tilde G$ is given by
		\begin{equation*}
			\begin{split}
				\tilde G((x_0, x_1, x_2, \ldots))&=H_2(G_1((x_0,x_n,x_{2n},\ldots)))\\
				&= H_2((h^{b-a}\circ \varphi^m(x_n), h^{b-a}\circ\varphi^m(x_{2n}), \ldots))\\
				&=(h^{b-a}\circ\varphi^m(x_n), h^{b-a}\circ\varphi^{m+n-1}(x_{2n}), \ldots, h^{b-a}\circ\varphi^{m+1}(x_{2n}), h^{b-a}\circ\varphi^m(x_{2n}), \ldots)\\
				&=(h^{b-a}(x_{n-m}), h^{b-a}(x_{n-m+1}), h^{b-a}(x_{n-m+2}), \ldots).
			\end{split}
		\end{equation*}
		Note that, since $h$ is a homeomorphism which commutes with $\varphi$, then $h^{b-a}$ and $\varphi$ strongly commute. To compute the entropy of $\tilde G$, we use Lemma~\ref{lem:shiftdiagonal} with $f=\varphi$, $g=h^{b-a}$, and $k=n-m$. We conclude that $\Ent(G)=\Ent(\tilde G)=\Ent(h^{b-a}\circ \varphi^{m-n})$.
	\end{proof}

	\begin{question}
		Can we, in general, given a homeomorphism $h\colon X\to X$, and an onto map $\varphi\colon X\to X$ which commutes with $h$, compute $\Ent(h\circ\varphi)$ in terms of $h$ and $\varphi$?
	\end{question}
	
	\begin{corollary}
		Assume $\varphi, h\colon I\to I$ are commuting interval maps, such that $\varphi$ is piecewise monotone, and $h$ is a homeomorphism. Let $a,b\geq 0$, $n,m\in\N$, and $f,g\colon I\to I$ be maps such that $f=h^a\circ\varphi^n$, and $g=h^b\circ\varphi^m$. Let $G\colon\varprojlim(I,f)\to\varprojlim(I,f)$ be a map defined by
		$$G((x_0,x_1,x_2, \ldots))=(g(x_1), g(x_2),\ldots),$$
		for all $(x_0,x_1,x_2,\ldots)\in\varprojlim(I,f)$. Then
		$$\Ent(G)=|m-n|\cdot\Ent(\varphi).$$
	\end{corollary}
	\begin{proof}
		By Theorem~\ref{thm:ent}, we have $\Ent(G)=\Ent(h^{b-a}\circ\varphi^{m-n}).$ 
		
		If $m\geq n$, then $\varphi^{m-n}$ is a piecewise monotone map. Therefore, since $h^{b-a}$ is a homeomorphism, Lemma~\ref{lem:entropy} implies that $\Ent(h^{b-a}\circ\varphi^{m-n})=\Ent(\varphi^{m-n})$. Lemma~\ref{lem:power} then implies that $\Ent(\varphi^{m-n})=(m-n)\Ent(\varphi).$
		
		If $m<n$, then $\varphi^{n-m}$ is a piecewise monotone map. Since $\varphi^{m-n}=(\varphi^{n-m})^{-1}$, and $h$ and $\varphi$ commute, we have $\Ent(h^{b-a}\circ\varphi^{m-n})=\Ent((h^{a-b}\circ\varphi^{n-m})^{-1})=\Ent(h^{a-b}\circ\varphi^{n-m})$, where the last equality follows from Lemma~\ref{lem:inverse}. Again by Lemma~\ref{lem:entropy} and Lemma~\ref{lem:power}, we have $\Ent(h^{a-b}\circ\varphi^{n-m})=\Ent(\varphi^{n-m})=(n-m)\Ent(\varphi)$.
		
		Thus, $\Ent(G)=|m-n|\cdot\Ent(\varphi)$.
	\end{proof}
	
	The next example comes from the Senior Thesis of Anna Cole \cite{Coleslaw} and motivates the study of commuting piecewise monotone interval maps and homeomorphisms. We will continue this study in an upcoming paper \cite{AM-preprint}.
	
	\begin{example}\label{ex:Anna} 
		Let \[\varphi(x)= \begin{cases}
			2x & \mbox{ if $ 0\leq x< \frac{1}{2} $}\\
			\frac 32-x & \mbox{ if $ \frac{1}{2} \leq x  \leq 1 $},\\
		\end{cases}
		\]
		and let $\eta\colon\left[\frac 12,1\right]\to\left[\frac 12,1\right]$ be any homeomorphism for which $\eta(\frac 12)=\frac 12$, $\eta(1)=1$, and $\eta(\frac 32-x)=\frac 32-\eta(x)$ for all $x\in\left[\frac 12,1\right]$.
		
		We define a homeomorphism $h\colon I\to I$ as:
		\[h(x)= \begin{cases}
			0 & \mbox{ if $ x=0$}\\
			
			\frac{\eta(2^nx)}{2^n} & \mbox{ if $ \frac{1}{2^{n+1}} \leq x \leq \frac{1}{2^n} $ and $ n\geq 0$.}
		\end{cases}
		\]
		
		See Figure~\ref{fig:Anna}.
		
		We now show that $\varphi$ and $h$ commute. First, $\varphi\circ h(0)=0=h\circ \varphi(0)$. Next, assume $x\in\left[\frac 12,1\right]$, so $h(x)=\eta(x)\in\left[\frac 12,1\right]$, and $\varphi(x)=\frac 32-x\in\left[\frac 12,1\right]$. Thus $\varphi\circ h(x)=\varphi(\eta(x))=\frac 32-\eta(x)$, and $h\circ \varphi(x)=\eta(\frac 32-x)$. Since we assumed $\eta(\frac 32-x)=\frac 32-\eta(x)$ for all $x\in\left[\frac 12,1\right]$, we get $\varphi\circ h(x)=h\circ \varphi(x)$ for all $x\in\left[\frac 12,1\right]$. Finally, let $n\geq 1$, and $x\in\left[\frac{1}{2^{n+1}},\frac{1}{2^n}\right]$. Then $h(x)=\frac{\eta(2^nx)}{2^n}\in\left[\frac{1}{2^{n+1}},\frac{1}{2^n}\right]$, and $\varphi(x)=2x\in\left[\frac{1}{2^n},\frac{1}{2^{n-1}}\right]$. Thus, $\varphi\circ h(x)=\frac{\eta(2^nx)}{2^{n-1}}=\frac{\eta(2^{n-1}\cdot 2x)}{2^{n-1}}=h(2x)=h(\varphi(x))$.
		
		Let $a,b\geq 0$, and $n,m\in\N$, and let $f=h^a\circ \varphi^n$, and $g=h^b\circ\varphi^m$. Then $f$ and $g$ are kin (which are not cousins unless $a=b=0$). Given a map $G\colon\varprojlim(I,f)\to\varprojlim(I,f)$, $G((x_0,x_1,x_2,\ldots))=(g(x_1),g(x_2),\ldots)$, we get
		$$\Ent(G)=|m-n|\cdot\Ent(\varphi).$$
		Since $\Ent(\varphi)=0$, we get $\Ent(G)=0$.
	\end{example}
	
	\begin{figure}[ht!]
		\begin{tikzpicture}[scale=8]
			\draw (0,0)--(0,1)--(1,1)--(1,0)--(0,0);
			\draw[dashed, thin] (0,0)--(1,1);
			
			\newcommand*\List{{1/256, 1/128, 1/64, 1/32, 1/16, 1/8, 1/4, 1/2, 1}}
			\foreach \i in {0,...,6} {
				\draw[dashed] (\List[\i],\List[\i+1]) rectangle (\List[\i+1],\List[ \i+2]);
			}			\foreach \i in {0,...,7} {
				\draw[dashed] (\List[\i],\List[\i]) rectangle (\List[\i+1],\List[ \i+1]);
			}
			
			\draw[red,thick] (0,0)--(0.5,1)--(1,0.5);
			
			\foreach \i in {0,...,7}{
				\pgfmathparse{\List[\i]}
				\let\b\pgfmathresult		
				\pgfmathparse{\List[\i+1]}
				\let\c\pgfmathresult
				\pgfmathparse{(\c - \b)/4}
				\let\r\pgfmathresult
				\draw[thick, blue] (\List[\i], \List[\i])--(\List[\i]+\r, \List[\i]+1/3*\List[\i])--(\List[\i]+\r+\r+\r, \List[\i]+2/3*\List[\i])--(\List[\i+1], \List[\i+1]);
			}	
		\end{tikzpicture}
		\caption{Maps $\varphi\colon I\to I$ (in red) and $h\colon I\to I$ (in blue) from Example~\ref{ex:Anna}. Here, $\eta\colon[\frac 12,1]\to[\frac 12,1]$ is given by $\eta(\frac 12)=\frac 12$, $\eta(\frac 58)=\frac 23$, $\eta(\frac 78)=\frac 56$, $\eta(1)=1$, and extended by linearity. Note that $\eta(\frac 32-x)=\frac 32-\eta(x)$ for all $x\in[\frac 12,1]$, and $\eta$ is a homeomorphism.}
		\label{fig:Anna}
	\end{figure}
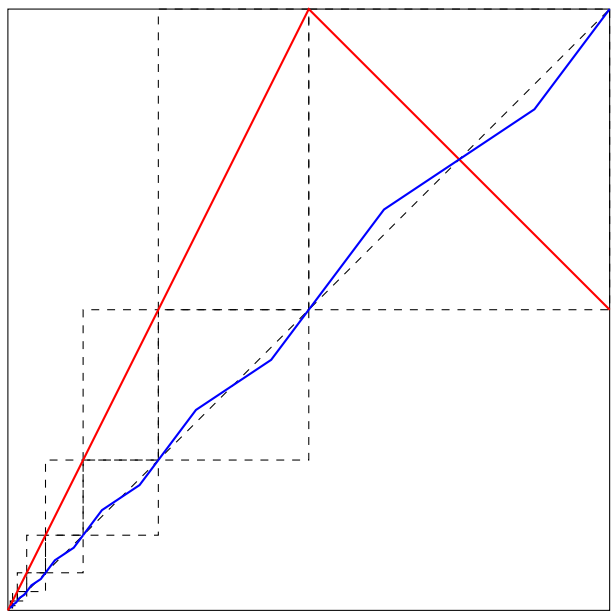

\begin{thebibliography}{A}
		
		\bibitem{Adler} R.\ Adler, A.\ Konheim,  M.\ McAndrew,   {\em  Topological Entropy}, Trans. A.M.S. {\bf 114} (2) (1965), 309--319. 
		
		\bibitem{AliKoo} A.\ Alikhani-Koopaei, {\em On common fixed points, periodic points and recurrent points of continuous functions}, Int.\ J.\ Math.\ {\bf 39} (2003) 2465--2473.
		
		\bibitem{AM-entropy} A.\ Anu\v si\'c, C.\ Mouron, {\em Topological entropy of diagonal maps on inverse limit spaces}, Topol. Methods Nonlinear Anal. {\bf 59} (2022), 867--895. 
		
		\bibitem{AM-commuting} A.\ Anu\v si\'c, C.\ Mouron, {\em Strongly commuting interval maps}, Fund.\ Math. {\bf 257} (2022), 39--68.
		
		\bibitem{AM-preprint} A.\ Anu\v si\'c, C.\ Mouron, {\em Interval maps which commute with homeomorphisms}, preprint 2025.
		
		\bibitem{Baxter} G.\ Baxter, {\em On fixed points of the composite of commuting functions}, Proc.\ Amer.\ Math.\ Soc.\ {\bf 15} (1964), 851--855.
		
		\bibitem{Bellamy80} D.\ P.\ Bellamy, {\em  tree-like continuum without the fixed-point property}, Houston J.\ Math.\ {\bf 6} (1980), no. 1, 1--13.
		
		\bibitem{Bing69} R.\ H.\ Bing, {\em The elusive fixed point property}, Amer.\ Math.\ Monthly {\bf 76} (1969), 119--132.
		
		\bibitem{Boronski} J.\ Boronski, {\em A note on fixed points of abelian actions in dimension one}, Proc.\ Amer.\ Math.\ Soc.\ {\bf 147} (2019) (4), 1653--1655. 
		
		\bibitem{Boyce} W.\ M.\ Boyce, {\em Commuting functions with no common fixed point}, Trans.\ Amer.\ Math.\ Soc.\ {\bf 137} (1969), 77--92.
		
		\bibitem{RFBrown} R.\ F.\ Brown, {\em A good question won't go away: An example of mathematical research}, American Mathematical Monthly, \black{ {\bf 128(1)} (2021), 62--68.}
		
		\bibitem{CanovasLinero} J.\ C\'anovas, A.\ Linero, {\em On the dynamics of compositions of commuting interval maps}, J.\ Math.\ Anal.\ Appl.\ {\bf 305} (2005), 296--303.
		
		\bibitem{Cohen} H.\ Cohen, {\em On fixed points of commuting functions}, Proc.\ Amer.\ Math.\ Soc.\ {\bf 15} (1964), 293--296.
		
		\bibitem{Coleslaw} A.\ Cole, {\em Commutativity of nonopen and open maps}. Senior Thesis, Rhodes College (2022).
		
		\bibitem{Folkman} J.\ H.\ Folkman, {\em On functions that commute with full functions}, Proc.\ Amer.\ Math.\ Soc.\ {\bf 17} (1966), 383--386.
		
		\bibitem{GrincSnoha} M.\ Grin\v c, L.\ Snoha, {\em Jungck theorem for triangular maps and related results}, Appl.\ Gen.\ Topol.\ {\bf 1} (1) (2000) 83--92.
		
		\bibitem{Hagopian07} C. L. Hagopian, {\em An update on the elusive fixed-point property}, Open problems in topology II (Elliott Pearl, ed.), Elsevier B. V., Amsterdam, 2007, pp. 263--277.
		
		\bibitem{HoHerGut} L.\ Hoehn, R.\ Hern\'andez-Guti\'errez, {\em A fixed-point-free map of a tree-like continuum induced by bounded valence maps on trees}, Colloq.\ Math.\ {\bf 151} (2018), no. 2, 305--316.
		
		\bibitem{Huneke} J.\ P.\ Huneke, {\em On common fixed points of commuting continuous functions on an interval}, Trans.\ Amer.\ Math.\ Soc.\ {\bf 139} (1969), 371--381.
		
		\bibitem{Isbell} J.\ R.\ Isbell, {\em Research problems: Commuting wrappings of trees}, Bull.\ Amer.\ Math.\ Soc.\ {\bf 63} (1957), no. 6, 419.
		
		\bibitem{Joichi} J.\ T.\ Joichi, {\em On functions that commute with full functions and common fixed points}, Nieuw Arch.\ Wisk.\ (3) {\bf 14} (1966), 247--251.
		
		\bibitem{KelTen} J.\ P.\ Kelly, T.\ Tennant, {\em Topological entropy of set-valued functions}, Houston J.\ Math.\ {\bf 43} (1) (2017), 263--282.
		
		\bibitem{Linero} A.\ Linero, {\em Common fixed points for commuting Cournot maps}, Real Anal.\ Exchange {\bf 28} (1) (2002/3), 121--143.
		
		\bibitem{McDowell} E.\ L.\ McDowell, {\em Coincidence Values of Commuting Functions}, Top.\ Proc.\ {\bf 34} (2009), 365--384.
		
		\bibitem{MisSl} M.\ Misiurewicz, W.\ Szlenk, {\em Entropy of piecewise monotone mappings}, Studia Math.\ {\bf 67} (1980), 45--63. 
		
		\bibitem{Mou} C.\ Mouron, {\em Dynamics of commuting homeomorphisms of chainable continua}, Colloq.\ Math.\ {\bf 121} (2010), 63--77.
		
		\bibitem{OvRog} L.\ G.\  Oversteegen, J.\ T.\ Rogers, Jr., {\em An  inverse  limit description of an atriodic tree-like continuum and an induced map without a fixed point}, Houston J.\ Math.\ {\bf 6} (1980), no. 4, 549--564.
		
		\bibitem{OvRog2} L.\ G.\  Oversteegen, J.\ T.\ Rogers, Jr., {\em Fixed-point-free maps on tree-like continua}, Topology Appl.\ {\bf 13} (1982), no. 1, 85--95. 
		
		\bibitem{Ritt} J.\ F.\ Ritt, {\em Permutable rational functions}, Trans.\ Amer.\ Math.\ Soc.\ {\bf 25} (1923), no. 3, 399--448.
		
		\bibitem{Walters} P.\ Walters, {\em An Introduction to Ergodic Theory}, Graduate Texts in Math. {\bf 79} (1982) Springer, New
		York.
		
		\bibitem{Ye} X.\ Ye, {\em Topological entropy of the induced maps of the inverse limits with bonding maps}, Topology Appl.\ {\bf 67} (1995), 113--118.
		
	\end{thebibliography}
\end{document}